\documentclass[11pt,reqno]{amsart}
\usepackage{amsmath}
\usepackage{amssymb}
\usepackage{mathrsfs}
\usepackage{mhequ}
\usepackage{color}
\usepackage{graphicx}
\usepackage{enumerate}
\usepackage{todonotes}
\usepackage[unicode=true,
 bookmarks=true,bookmarksnumbered=false,bookmarksopen=false,
 breaklinks=false,pdfborder={0 0 1},backref=false,colorlinks=true]{hyperref}
\usepackage{bbm} 
\usepackage[small]{caption}
\usepackage{hyphenat}

\newtheorem{maintheorem}{Theorem}

\newtheorem{theorem}{Theorem}[section] 
\newtheorem{lemma}[theorem]{Lemma}     
\newtheorem{corollary}[theorem]{Corollary}
\newtheorem{proposition}[theorem]{Proposition}

\newtheorem{definition}[theorem]{Definition}

\newtheorem{remark}[theorem]{Remark}

\newtheorem{example}[theorem]{Example}

\newcommand{\N} {\mathbb N}

\newcommand{\figref}[1]{Figure~\ref{#1}}

\newtheorem*{problem1}{Monge Transport Problem}
\newtheorem*{problem2}{Monge-Kantorovich Transport Problem}

\DeclareMathOperator{\vol}{vol}

\begin{document}

	\title[Geometric properties of disintegration of measures]
		 {Geometric properties of disintegration of measures}

	\author[R. Possobon]{Renata Possobon}
	\address{Renata Possobon\\
		Institute of Mathematics, Department of Applied Mathematics\\
		Universidade Estadual de Campinas\\
 	 	13.083-859 Campinas - SP\\
 	 	Brazil}
   \email{re.possobon@gmail.com}

	\author[C.S. Rodrigues]{Christian S.~Rodrigues}
	\address{Christian S.~Rodrigues\\
		Institute of Mathematics, Department of Applied Mathematics\\
		Universidade Estadual de Campinas\\
		13.083-859 Campinas - SP\\
   		Brazil\\
		}
  	\email{rodrigues@ime.unicamp.br}

	\date{\today}

	\begin{abstract} 
		\nohyphens{In this paper, we study a connection between disintegration of measures and geometric properties of probability spaces. We prove a disintegration theorem, addressing disintegration from the perspective of an optimal transport problem. We look at the disintegration of transport plans, which are used to define and study disintegration maps. Using these objects, we study the regularity and absolute continuity of disintegration of measures. In particular, we exhibit conditions for which the disintegration map is weakly continuous and one can obtain a path of measures given by this map. We show a rigidity condition for the disintegration of measures to be given into absolutely continuous measures.}
	\end{abstract}



	\maketitle
	

	\section{Introduction}

		The disintegration of a measure over a partition of the space on which it is defined is a way to rewrite this measure as a combination of probability measures, which are concentrated on the elements of the partition. As an example, consider the probability space $(X, \mathcal{F}, \mu)$ and its partition into a finite number of measurable subsets $P_{1}, \dots, P_{n}$ with positive measure. A disintegration of $\mu$ with respect to (w.r.t.) this partition is a family of probabilities $\{\mu_1, \dots \mu_n \}$ on $X$, such that, for $i=1, \dots, n$, we have $\mu_{i} (P_i)=1$ and, for every measurable set $E \subset X$, the conditional measures are given by $\mu_i (E) = {\mu(E \cap P_i)}\slash{\mu(P_i)}$. It is possible to write the original measure as a combination of the conditional ones
		\begin{displaymath}
			\mu(E) = \sum_{i=1}^{n} \mu (P_{i}) \mu_{i}(E) = \sum_{i=1}^{n} \mu (P_{i}) \frac{\mu(E \cap P_i)}{\mu(P_i)}.
		\end{displaymath}
	
		More generally, consider a probability space $(X, \mathcal{F}, \mu)$ and a partition $\mathcal{P}$ of $X$ into measurable subsets. Let $\varrho$ be the  natural projection that associates each point $x \in X$ to the element $P \in \mathcal{P}$ which contains $x$. The measurable function $\varrho$ can be used to induce a probability $\hat{\mu}$ on $\mathcal{P}$. A subset $B \subset \mathcal{P}$ is measurable if, and only if, $\varrho^{-1}(B)$ is a measurable subset of $X$. Then, the family $\hat{\mathcal{B}}(\mathcal{P})$ of measurable subsets is a $\sigma$-algebra on $\mathcal{P}$. Let $\hat{\mu}$ denote the measure given by
		\begin{displaymath}
			\hat{\mu}(B) = \varrho_{*}\mu (B) := \mu \circ \varrho^{-1} (B) = \mu(\{x \in X : \varrho(x) \in B\})
		\end{displaymath}
		for every $B \in \hat{\mathcal{B}}(\mathcal{P})$. In this case, $\hat{\mu}$ is called the \textbf{law} of $\varrho$, denoted by $\operatorname{law}(\varrho)$. A disintegration of $\mu$ with respect to $\mathcal{P}$ into conditional measures is a family $\{{\mu}_{P}:P\in \mathcal{P} \}$ of probability measures on $X$, such that, for every $E \in \mathcal{F}$,
		\begin{enumerate}
			\item $\mu_{P}(P)=1$ for $\hat{\mu}$-almost every $P \in \mathcal P$;
			\item $P \mapsto \mu_{P}(E)$ is measurable;
			\item $\mu(E)=\int \mu_{P}(E) ~d \hat{\mu}(P)$.
		\end{enumerate}
		
		There are several reasons why one may wish to study such possible combinations of measures. In Ergodic Theory, for example, the disintegration of a measure is directly related to the ergodic decomposition of invariant measures, which are crucial objects encoding the asymptotic behaviour of dynamical systems \cite{OV14}. The concept of disintegration, however, appears in much broader context in areas, such as Probability \cite{CP97, Par67}, Geometry \cite{Stu06}, among others \cite{BM17, GL20, Var16, Vil03}.
	
		The idea of disintegrating a measure was devised by Von Neumann in 1932 \cite{Von32}. Since then, different versions of disintegration theorems have been presented and used, for example in~\cite{AGS05, AP03, DM78, Tue75}, to name a few. In particular, the well-known Rokhlin's Disintegration Theorem shows that there exists a disintegration of $\mu$ relative to $\mathcal{P}$ if $X$ is a complete separable metric space and $\mathcal{P}$ is a measurable partition. By measurable partition we mean that there exists some measurable set $X_{0} \subset X$, such that $\mu(X)=\mu(X_{0})$, and $\mathcal{P} = \bigvee_{n=1}^{\infty} \mathcal P_{n} = \{ P_1 \cap P_2 \cap \dots : P_n \in \mathcal{P}_n, ~\text{for all}~ n \geq 1 \}$ restricted to $X_{0}$, for an increasing sequence $\mathcal P_{1} \prec \mathcal P_{2} \prec \dots \prec \mathcal P_{n} \prec \dots$ of countable partitions \cite{Rok52}.
	
		Recently, Simmons has proposed a more general and subtle formulation of Rokhlin's Disintegration Theorem, where he has considered any universally measurable space $(X, \mathcal{B}, \mu)$ and  a measure space $Y$, for which there exists an injective map $Y \to \{0, 1\}^{\N}$. That is, $Y$ is any subspace of the standard Borel space. He has shown that there is a (unique) system of conditional measures $(\mu_{y})_{y \in Y}$, a disintegration of $\mu$~\cite{Sim12}. Then, his formulation is further developed to address $\sigma$-finite measure spaces with absolutely continuous morphisms. One of the facts standing out in Simmons' formulation is a fibre-wise perspective, which we wish to further explore.

		Even though, geometric properties of the space where a measure is defined and statistical properties obtained via disintegration theorems seem to be strictly connected, for instance in foliated manifolds, very little geometric information is taken into account while studying disintegration of measures. In particular, intrinsic geometric properties of probability spaces are very often neglected. The purpose of this paper is to advertise the viewpoint of tackling disintegration of measures taking into consideration intrinsic structures of probability spaces obtained via Optimal Transport Theory. To this end, we will formulate disintegration of probability measures in terms of a transportation problem to explore a fibre-wise formulation of a disintegration and its consequences.

		\subsection*{Main results}
		
			As a first result in this paper, we proved a disintegration theorem, Theorem~\ref{conditional}, and we introduced a fibre-wise perspective on disintegration. Using this disintegration theorem, we studied conditions in which one can obtain a path of conditional measures in the space of probability measures. In particular, in the propositions \ref{weakly_continuous}, \ref{bijective} and \ref{mmf}, we investigated the weak continuity of the disintegration map, which parametrises the conditional measures. The last result in this paper is Theorem \ref{thm.absolute}. Its first part shows how to construct a path of conditional measures in the space of probabilities. Its second part gives us a sort of rigidity result for disintegration of measures. Namely, we showed that if one of the measures in this path is absolutely continuous, then all measures in the associated path must also be absolutely continuous. The third part is a particular case of absolutely continuity in which the disintegration map is an isometry.
			
			The paper is organised as follows. Section \ref{sec.settings} contains the main concepts from Optimal Transport Theory to be used along the paper. In Section~\ref{sec.disintegra}, a disintegration theorem, stated as Theorem~\ref{conditional}, is proved. In Section~\ref{sec.disint-maps}, Theorem~\ref{conditional} is used to define and study what we called disintegration maps. These crucial objects are used in Section~\ref{sec.absolute}. There we proved a series of propositions about a disintegration map: in Proposition \ref{weakly_continuous} we show that this map is nearly weakly continuous, under some assumptions on the reference measure $\nu$. In propositions \ref{bijective} and \ref{mmf} we studied hypotheses about the disintegration for which this map is weakly continuous. Afterwards we proved our main result in this section, Theorem \ref{thm.absolute}, about paths of measures given via disintegration maps and a rigidity condition establishing absolute continuity of measures in these paths.
		

\
		
	\section{Spaces of probability measures, Wasserstein spaces and Optimal Transport}\label{sec.settings}
	
		In this section, we set up some notation to be used throughout the text. We also introduce some basic terminology from Optimal Transport Theory, which is meant for readers not familiar with this area. Those who are skilled on the topic may wish to skip this section. Our main references for this section are~\cite{AGS05, Amb00, Vil09}.
	
		The cornerstone for the theory of optimal transportation is considered to be a logistic problem addressed by Gaspar Monge in 1781. The main idea is to transport masses from a given location to another one at minimal cost. To state it in a modern formulation, let $\mathscr{P}(X)$ be the set of all Borel probability measures on $X$. The problem amounts to the following.
	
		\
	
 		\begin{problem1}\label{Monge-problem}
			Let $X$, $Y$ be Radon spaces. Given measures $\mu \in \mathscr{P}(X)$, $\nu \in \mathscr{P}(Y)$, and a fixed Borel cost function $c: X \times Y \to [0, \infty]$, minimize
			\begin{displaymath}
				T \mapsto \int_{X} c(x, T(x)) ~d\mu
			\end{displaymath}
			among all maps $T$, such that, $T_{*}\mu = \nu$.  		
		\end{problem1}

		\

		\noindent The maps $T$ fulfilling  $T_{*}\mu = \nu$ are called \textbf{transport maps}. The Monge Transport Problem actually may be ill-posed and such a map does not need to exist. That is the case, for example, when one of the measures is a Dirac mass and the other one is not. A way around is given by a different formulation as proposed by Kantorovich.
	
		\
	
		\begin{problem2}
			Let $X$, $Y$ be Radon spaces. Given measures $\mu \in \mathscr{P}(X)$, $\nu \in \mathscr{P}(Y)$, and a fixed Borel cost function $c: X \times Y \to [0, \infty]$, minimize
			\begin{equation} \label{Kantorovich}
				\gamma \mapsto	\int_{X \times Y} c(x, y) ~d\gamma(x, y)
			\end{equation}
			among all measures $\gamma \in \mathscr{P}(X \times Y)$  with \textbf{marginals} $\mu$ and $\nu$, i. e., $\gamma$ satisfying $(\text{proj}_{X})_{*}\gamma=\mu$ and $(\text{proj}_{Y})_{*}\gamma=\nu$, where $\text{proj}_{X}$ and $\text{proj}_{Y}$ are the canonical projections $(x, y) \mapsto x$ and $(x, y) \mapsto y$, respectively. 
		\end{problem2}

		\
	
		\noindent The measures $\gamma \in \mathscr{P}(X \times Y)$ are called \textbf{transport plans}. We denote the set of all transport plans with marginals $\mu$ and $\nu$ by $\Pi(\mu, \nu)$. The value $C(\mu, \nu) = \inf_{\gamma \in \Pi(\mu, \nu)} \int_{X \times Y} c(x, y) ~d\gamma(x, y)$ is called \textbf{optimal cost}. Note that $(\text{proj}_{X})_{*}\gamma=\mu$ is equivalent to $\gamma[A \times Y]=\mu[A]$, for every $A \in \mathcal{B}(X)$, and $(\text{proj}_{Y})_{*}\gamma=\nu$ is equivalent to $\gamma[X \times B]=\nu[B]$, for every $B \in \mathcal{B}(Y)$. Moreover, it is possible to describe this problem in terms of the coupling of measures, in the following sense:
		
		\
		
		\begin{definition} \label{coupling}
			Let $(X, \mu)$ and $(Y, \nu)$ be probability spaces. A \textbf{coupling} of $(\mu, \nu)$ is a pair $(\mathcal{X}, \mathcal{Y})$ of measurable functions in a probability space $(\Omega, \mathbbm{P})$, such that,  $\operatorname{law}(\mathcal{X})=\mu$ and $\operatorname{law}(\mathcal{Y})=\nu$.
		\end{definition}
		
		\

		\noindent Considering $\Omega = X \times Y$, coupling $\mu$ and $\nu$ means to construct $\gamma \in \mathscr{P}(X \times Y)$ with marginals $\mu$ and $\nu$. Then, the Monge-Kantorovich Transport Problem can be understood as the minimization of the total cost, over all possible couplings of $(\mu, \nu)$.

		Whenever these problems are stated in metric spaces, we may choose the cost function to be the distance function itself, which in turn allows us to introduce a distance function between measures.
	
		\
		
		\begin{definition} [Wasserstein distance]
			Let $(X, d)$ be a separable complete metric space. Consider probability measures $\mu$ and $\nu$ on $X$ and  $p \in [1, \infty)$. The \textbf{Wasserstein distance} of order $p$ between $\mu$ and $\nu$ is given by
			\begin{displaymath}
				W_{p}(\mu, \nu):= \Big( \inf\limits_{\gamma \in \Pi(\mu, \nu)}  \int d(x_{1}, x_{2})^{p} ~d\gamma (x_{1}, x_{2})\Big)^{\frac{1}{p}}.
			\end{displaymath}
		\end{definition}
	
		\
	
		\noindent In general, $W_{p}$ is not a distance in the strict sense, because it can take the value $+\infty$. To rule this situation out, it is natural to constrain $W_{p}$ to a subset in which it takes finite values.
	
		\
	
		\begin{definition}[Wasserstein space]
			The \textbf{Wasserstein space} of order $p$ is defined by
			\begin{displaymath}
				\mathscr{P}_{p}(X):= \Big\{ \mu \in \mathscr{P}(X) : \int d(x, \tilde{x})^{p} \mu(dx) < +\infty \Big\}
			\end{displaymath}
			for $\tilde{x} \in X$ arbitrary. 
		\end{definition}
	
		\
		
		\noindent Therefore, $W_{p}$ sets a (finite) distance on $\mathscr{P}_{p}(X)$. It turns out that if $(X, d)$ is a complete separable metric space and $d$ is bounded, then the $p$-Wasserstein distance metrizes the weak topology over $\mathscr{P}(X)$~\cite[Corollary 6.13]{Vil09}. Furthermore, if $X$ is a complete metric space, then so is $\mathscr{P}(X)$ with the $p$-Wasserstein distance~\cite[Theorem 6.18]{Vil09}.
	
		Although the Wasserstein distance is defined for every $p \geq 1$, in this paper we will chose either $p=1$ or $p=2$, as stated later on. The reason is that for $W_{1}$ we can explicitly compute distance bounds, while for $W_{2}$ the space $\mathscr{P}_{2}(X)$ inherits geometric properties of the space $X$. The choice is always indicated throughout the text.
	
		We finish this section recalling two well-known results for later use. The first one is regarding measurable and continuous functions, the Lusin Theorem in a specific form. The other one gives us a ``recipe'' to glue different couplings.
	
		\
	
		\begin{theorem}~\cite[Theorem 2.3.5]{Fed69}.
			\label{thm.Lusin}
			Let $M$ be a locally compact metric space and $N$ a separable metric space. Consider $\mu$ a Borel measure on $M$, $A \subset M$ a measurable set with finite measure and $f : M \to N$ a measurable map. Then, for each $\delta > 0$ there is a closed set $K \subset A$, with $\mu(A \backslash K) < \delta$, such that the restriction of $f$ to $K$ is continuous.
		\end{theorem}
	
		\
	
		\begin{lemma}[Gluing Lemma]~\cite[Chapter 1]{Vil09} \label{gluing} 
			Let $(X_{i}, \mu_{i})$, $i = 1, 2, 3$, be complete separable metric probability spaces. If $(\mathcal{X}_1,\mathcal{X}_2)$ is a coupling of $(\mu_{1}, \mu_{2})$ and $(\mathcal{Y}_2, \mathcal{Y}_3)$ is a coupling of $(\mu_{2}, \mu_{3})$, then one can construct a triple of random variables $(\mathcal{Z}_{1}, \mathcal{Z}_{2}, \mathcal{Z}_{3})$ such	that $(\mathcal{Z}_{1}, \mathcal{Z}_{2})$ has the same law as $(\mathcal{X}_1,\mathcal{X}_2)$ and $(\mathcal{Z}_{2}, \mathcal{Z}_{3})$ has the same law as $(\mathcal{Y}_2, \mathcal{Y}_3)$. If $\mu_{12}$ stands for the law of $(\mathcal{X}_1,\mathcal{X}_2)$ on $X_1 \times X_2$ and $\mu_{23}$ stands for the law of $(\mathcal{X}_2,\mathcal{X}_3)$ on $X_2 \times X_3$, then to construct the joint law $\mu_{123}$ of $(\mathcal{Z}_{1}, \mathcal{Z}_{2}, \mathcal{Z}_{3})$ one just has to glue $\mu_{12}$ and $\mu_{23}$ along their common marginal $\mu_{2}$.
		\end{lemma}

\

	\section{Disintegration of measures}\label{sec.disintegra}
	
		In order to grasp some of the properties of the probability spaces while studying disintegration of measures, we would like to associate the latter with the Optimal Transport Theory. Before doing so, we prove the following disintegration theorem. Our proof is based on the idea of choosing a dense subset of a vector space using separability. Then, we extend a tailored made linear functional to the whole space, which is implicitly used in the proof of Dellacherie and Meyer~\cite[III-70]{DM78}, although our conditions are different.
	
		\
	
		\begin{maintheorem} \label{conditional}
			Let $X$ and $Y$ be locally compact and separable metric spaces. Let $\pi: X \to Y$ be a Borel map, and take $\mu \in \mathcal{M}_{+}(X)$, where $\mathcal{M}_{+}(X)$ is the set of all positive and finite Radon measures on $X$. Define $\nu = \pi_{*}\mu$ in $\mathcal{M}_{+}(Y)$. Then, there exist measures $\mu_{y} \in \mathcal{M}_{+}(X)$, such that,
			\begin{enumerate}
				\item $y \mapsto \mu_{y}$ is a Borel map and $\mu_{y} \in \mathscr{P}(X)$ for $\nu$-almost every $y \in Y$;
				\item $\mu = \nu \otimes \mu_{y}$, that is, $\mu(A)= \int_{Y} \mu_{y} (A) ~d\nu(y)$ for every $A \in \mathcal{B}(X)$;
				\item $\mu_{y}$ is concentrated on $\pi^{-1}(y)$ for $\nu$-almost every $y \in Y$.
			\end{enumerate}
		\end{maintheorem}

		\begin{proof} 
			We shall first consider the disintegration of measures on compact metric spaces. Then, we tackle the general case as it is stated.\\

			\noindent \textbf{Step 1:} To get started, consider $X$ to be a compact metric space with its Borel $\sigma$-algebra $\mathcal{B}(X)$, and let $\mu$ be a Radon measure on $X$. If $(Y, \mathcal{E})$ is a measurable space, we define a measurable map $q: \mathcal{B}(X) \to \mathcal{E}$, so that we set $\nu=q_{*}\mu$. Let $C(X)$ be the set of all continuous real functions $\omega :X \to \mathbb{R}$. Then, for each $\omega \in C(X)$, we associate a measure $\lambda$ given by
			\begin{displaymath}
				\lambda(A)= \int\limits_{q^{-1}(A)} \omega (x) ~d\mu (x)
			\end{displaymath}
			for every $A \in \mathcal{E}$. The measure $\lambda$ is absolutely continuous with respect to $\nu$. Indeed, for every $A \in \mathcal{E}$ we have that $\nu(A) = \mu(q^{-1}(A)) = 0$ implies $\lambda(A)=0$. Therefore, since $\nu$ and $\lambda$ are positive measures and $\lambda \ll \nu$, by the Radon-Nikodym Theorem, there exists $h: Y \to [0, \infty]$, the density of $\lambda$ with respect to $\nu$, such that, $\lambda(A) = \int_{A} h ~d\nu$ for $A \in \mathcal{E}$.  Thus,
		
			\begin{equation}\label{eq.RNdensi}
				\int\limits_{A} h(y) ~d\nu(y) = \int\limits_{q^{-1}(A)} \omega(x) ~d\mu(x).
			\end{equation}
		
			Recall that  $C(X)$ is a separable space with the supremum norm. Let $\mathcal{H} = \{ \omega_1 \equiv 1, \omega_2, \omega_3, \dots \}$ be a dense subset of $C(X)$. Suppose, without loss of generality, that $\mathcal{H}$ is a vector space over $\mathbb{Q}$. Then, for each $n \in \mathbb{N}$ we consider the Radon-Nikodym density $h_{n}$ associated with $\omega_n$ given by~(\ref{eq.RNdensi}), so that, for each $A \in \mathcal{E}$,
			\begin{equation}\label{eq.RNdensIndex}
				\int\limits_{A} h_{n}(y) ~d\nu(y) = \int\limits_{q^{-1}(A)} \omega_n(x) ~d\mu(x).
			\end{equation}
			Note that $h_{n} \geq 0$ almost always, if $\omega_n \geq 0$, since $\nu=q_{*}\mu$.
			\\
		
			\noindent \textbf{Step 2:} We will use the associated densities to construct a linear functional which will be extended using the Hahn-Banach Theorem. In order to do so, we denote by $\mathbb{A}$ the set of all $y \in Y$, such that, if $\omega_i= \alpha \omega_j + \beta \omega_k$, then we have for their associated densities that the relation $h_{i}(y) =  \alpha h_{j}(y) + \beta h_{k}(y)$ holds true, where $\alpha, \beta \in \mathbb{Q}$, and the associated density to $\omega_{1}$ is set to $h_{1}(y) = 1$. The set $\mathbb{A}$ is measurable and $\nu(\mathbb{A})=1$. Indeed,
		
			\begin{align*}
				\int_{q^{-1}(Y)} \omega_i(x) ~d\mu &= \int_{q^{-1}(Y)} (\alpha \omega_j + \beta \omega_k)(x) ~d\mu \\
				&= \alpha \Big( \int_{q^{-1}(Y)} \omega_j(x) ~d\mu  \Big) + \beta \Big( \int_{q^{-1}(Y)} \omega_k(x) ~d\mu \Big), \text{ which by~(\ref{eq.RNdensIndex}), }\\
				&= \alpha \Big( \int_Y h_{j}(y) ~d\nu \Big) + \beta \Big( \int_Y h_{k}(y) ~d\nu \Big) \\
				&= \int_Y \alpha  h_{j}(y) + \beta h_{k}(y) ~d\nu,
			\end{align*}
			so $ \int  h_{i}(y) ~d\nu = \int \alpha  h_{j}(y) + \beta h_{k}(y) ~d\nu$. Consequently, $h_{i}(y) =  \alpha h_{j}(y) + \beta h_{k}(y)$ $\nu$-almost always. Therefore, whenever $\omega_{n}$ is a linear combination of elements of $\mathcal{H}$, then the associated densities defined by~(\ref{eq.RNdensIndex}) can also be written point-wise as linear combinations of Radon-Nikodym densities.\\
		
			\noindent \textbf{Step 3:} For each $y \in \mathbb{A}$, we define the functional $\tilde{\varphi_y}: \mathcal{H} \to \mathbb{R}$, given by $\tilde{\varphi_y}(\omega_n):=h_{n}(y)$. Note that $\tilde{\varphi_y}: \mathcal{H} \to \mathbb{R}$ is $\mathbb{Q}$-linear with $\| \tilde{\varphi_y} \| \leq 1$ and, since $\tilde{\varphi_y} (1)=1$, we have actually that $\| \tilde{\varphi_y} \| = 1$. Therefore, by the Hahn-Banach Theorem, $\tilde{\varphi_y}$ can be extended to a continuous positive linear functional $\varphi_y: C(X) \to \mathbb{R}$, with $\| \varphi_y \| = 1$. Furthermore, the Riesz-Markov-Kakutani Representation Theorem assures that there exists a unique Radon measure $\mu_y$ on $X$, such that, $\varphi_y(\omega)= \int \omega ~d \mu_y$ for every $\omega \in C(X)$ and $\varphi_y (1)=1$. Note that $\int \omega ~d \mu_y \leq 1$, since $\| \varphi_y \| = 1$. Thus, we conclude that $\mu_y$ is a probability measure. Note also that $\mu_y$ is supported on $q^{-1} \{ y \in \mathbb{A}\}$. For $y \notin \mathbb{A}$, consider $\mu_y = 0$.\\
		
			\noindent \textbf{Step 4:} Observe that $y \mapsto \int_{X} \omega_n ~d\mu_y$ is $\mathcal{E}$-measurable for every $\omega_n \in \mathcal{H}$ and
			\begin{displaymath}
				\int_{Y} \int_{q^{-1}(y)} \omega_n(x) ~d\mu_y ~d\nu = \int_{Y} \varphi_y(\omega_n) ~d\nu = \int_{Y} h_{n}(y) ~d\nu =  \int_{X} \omega_n(x) ~d \mu.
			\end{displaymath}
			By the definition of $\mathcal{H}$, we have that for each $\omega \in C(X)$, there exists a sequence $(\omega_{i})_{i}$, with $\omega_i \to \omega$ uniformly. So, by uniform convergence, we have that $y \mapsto \int_{X} \omega ~d\mu_y$ is $\mathcal{E}$-measurable for every $\omega \in C(X)$. Furthermore,
			\begin{displaymath}
				\int_{Y} \int_{q^{-1}(y)} \omega(x) ~d\mu_y ~d\nu =  \int_{X} \omega(x) ~d \mu.
			\end{displaymath}
			The same holds true for any bounded and $\mathcal{B}(X)$-measurable function $\omega$. Indeed, denote by $\mathscr{C}$ the class of functions such that $y \mapsto \int_{X} \omega ~d\mu_y$ is $\mathcal{E}$-measurable and $\int_{Y} \int_{q^{-1}(y)} \omega(x) ~d\mu_y ~d\nu =  \int_{X} \omega(x) ~d \mu$. Note that $C(X) \subset \mathscr{C}$, from what was shown before. If $A \subset X$ is an open set, then $\mathbbm{1}_{A} \in \mathscr{C}$. So, let $\mathcal{D}$ be the set of all characteristic functions in $\mathscr{C}$. If $\mathbbm{1}_{A_n} \in \mathcal{D}$ for $n \in \mathbb{N}$, we have that $\mathbbm{1}_{\cup A_n} \in \mathcal{D}$. If $\mathbbm{1}_{A} \in \mathcal{D}$, we have that $\mathbbm{1}_{A^c} \in \mathcal{D}$. Thus, the class of measurable sets whose characteristic functions are in $\mathcal{D}$ is $\mathcal{B}(X)$. Therefore, the result follows by monotone convergence. This completes the proof of Theorem~\ref{conditional} for compact spaces. \\
		
			\noindent \textbf{Step 5:} Let $X$ be a locally compact and separable metric space. Then, the set of continuous real-valued functions with compact support on $X$, denoted by $C_c(X)$, is a vector space. Such a vector space can be seen as the union of the spaces $C_c(K_{i})$ of continuous functions with support on compact sets $K_{i}$. Since $\mu$ is a Radon measure on $X$, the map $\varphi: C_c(X) \to \mathbb{R}$, such that, $\omega \mapsto \int_{X} \omega(x) ~d\mu$, is a continuous positive linear map. Note also that $\mu$ is supported on a set $\tilde{\mathcal{K}}$, which is a countable union of compact subsets $K_{i} \subset X$. So, we can imbed $X$ into a compact metric space $\mathcal{K}$ and identify $\mu$ with a measure on $\mathcal{K}$ with support on $\tilde{\mathcal{K}}$, and construct the measures $\mu_y$ as above. Consider $\mu_{y}=0$ for $y$ such that $\pi^{-1}(y) \notin \tilde{\mathcal{K}}$. Thus, there exist probability measures $\mu_{y}$ on $X$, such that, each $\mu_{y}$ is supported on $\pi^{-1}(y)$ and, for every $\omega \in C_c(X)$,
			\begin{equation} \label{dis_eq}
				\int_{Y} \int_{\pi^{-1}(y)} \omega(x) ~d\mu_y ~d\nu =  \int_{X} \omega(x) ~d \mu.
			\end{equation}
		
			Since $Y$ is a locally compact and separable metric space and $\pi: X \to Y$ is a Borel map, then $y \mapsto \mu_y$ is a Borel map for $\nu$-almost every $y \in Y$. Furthermore, note that \eqref{dis_eq} is equivalent to say that $\mu(A)= \int_{Y} \mu_{y} (A) ~d\nu(y)$ for every $A \in \mathcal{B}(X)$ and $\mu_{y}$ is concentrated on $\pi^{-1}(y)$ for $\nu$-almost every $y \in Y$. This concludes the proof.
		\end{proof}
	
		\
	
		Many interesting examples arise when we consider Theorem \ref{conditional} for the case of product spaces with the Borel map $\pi$ as the canonical projection on the first component, as follows.
	
		\
	
		\begin{corollary} \label{desintegra_prod}
			Let $X$ and $Y$ be locally compact and separable metric spaces. Let $\text{proj}_{X}: X \times Y \to X $ be the canonical projection on the first component, take $\gamma \in \mathcal{M}_{+}(X \times Y)$ and set $\mu = {\text{proj}_{X}}_{*}\gamma \in \mathcal{M}_{+}(X)$. Then, there exist measures $\gamma_{x} \in \mathcal{M}_{+}(X \times Y)$, such that,
			\begin{enumerate}
				\item $x \mapsto \gamma_{x}$ is a Borel map and $\gamma_{x} \in \mathscr{P}(X \times Y)$ for $\mu$-a. e. $x \in X$;
				\item $\gamma = \mu \otimes \gamma_{x}$, i. e. , $\gamma(A)= \int_{X} \gamma_{x} (A) ~d\mu(x)$ for every $A \in \mathcal{B}(X \times Y)$;
				\item $\gamma_{x}$ is concentrated on $\text{proj}_{X}^{-1} (x)$ for $\mu$-almost every $x \in X$.
			\end{enumerate}
		\end{corollary}
	 
		\
	 
		In fact, since $\gamma_{x}$ is concentrated on $\text{proj}_{X}^{-1}(x)=\{x\} \times Y$, we can consider each $\gamma_{x}$ as a measure on $Y$, writing $\gamma(B)=\int_{X} \gamma_{x}(\{y : (x, y) \in B \}) ~d\mu$ for every $B \in \mathcal{B}(X \times Y)$, adding a different point of view to disintegration. The following example illustrates this case of disintegration of measures.
	
		\

		\begin{example} \label{example1}

			Consider a Solid Torus $S^{1} \times D^{2}$. Let $\mathcal{F}^{s} = \{\{x\} \times D^{2} \}_{x \in S^{1}}$ be a foliation of $S^{1} \times D^{2}$, as represented in Figure \ref{toro}. Given a measure $\gamma \in \mathscr{P}(S^{1} \times D^{2})$, let $\text{proj}_{S^1}: S^{1} \times D^{2} \to S^{1}$ be the canonical projection on the first component and set $\mu={\text{proj}_{S^1}}_{*}\gamma$. The Theorem \ref{conditional}, with $\pi := \text{proj}_{S^1}$, gives us a disintegration $\{ \gamma_{x} : x \in S^{1} \}$ of $\gamma$ along the leaves. Since the measures $\gamma_{x}$ are concentrated on ${\text{proj}}_{S^1}^{-1}(x)=\{x\} \times D^2$ for $\mu$-a.e. $x \in S^{1}$, we can consider each $\gamma_{x}$ as a measure on $D^{2}$. That is, we can define a probability on $D^2$ for each $x \in S^1$.
	
			\begin{figure}[h!]
				\centering
				\includegraphics[scale=0.5]{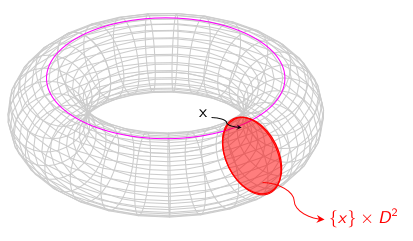}
				\caption{Representation of $\mathcal{F}^{s} = \{\{x\} \times D^{2} \}_{x \in S^{1}}$.} \label{toro}
			\end{figure}	
		\end{example}

		\
	
		The point of view from Corollary \ref{desintegra_prod} is, somehow, a generalisation of cases of disintegration of a probability measure along leaves in a foliated compact Riemannian manifold, as in Example \ref{example1}. In this sense, we remark that different versions of disintegration theorems can be related by taking suitable hypotheses, as we do in the following example.
	
		\
	
		\begin{example} \label{example_pro}
			Let $M_1$ and $M_2$ be compact Riemannian manifolds and set the product space $\Sigma = M_{1} \times M_{2}$. Let $\mathcal{F}^{s}=\{ \{x\} \times M_{2} \}_{x \in M_{1}}$ be a foliation of $\Sigma$. Given $\gamma \in \mathscr{P}(M_1 \times M_2)$, let ${\text{proj}}_{M_{1}} : \Sigma \to M_{1}$ be the canonical projection on $M_{1}$ and set $\mu={{\text{proj}}_{M_{1}}}_{*}\gamma$. By Theorem \ref{conditional}, there exists a family $\{ \gamma_{x} : x \in M_{1} \} \subset \mathscr{P}(M_{1} \times M_{2})$, such that,
			\begin{enumerate}
				\item $x \mapsto \gamma_{x}$ is a Borel map and $\gamma_{x} \in \mathscr{P}(M_{1} \times M_{2})$ for $\mu$-a. e. $x \in M_{1}$;
				\item $\gamma(A)= \int_{M_{1}} \gamma_{x} (A) ~d\mu(x)$ for every $A \in \mathcal{B}(M_{1} \times M_{2})$;
				\item $\gamma_{x}$ is concentrated on $\text{proj}_{M_{1}}^{-1} (x)$ for $\mu$-a. e. $x \in M_{1}$.
			\end{enumerate}
			Furthermore, we can consider each $\gamma_{x}$ as a probability on $M_{2}$, as stated above. Note that this result agrees with Rokhlin's Disintegration Theorem. In fact, let $\varrho: \Sigma \to \mathcal{F}^{s}$ be a map that associates each point $(x, y) \in \Sigma$ to the $\zeta$ element of $\mathcal{F}^{s}$ that contains $(x, y)$. Consider $\hat{\gamma}=\varrho_{*}\gamma$. Note that $\mathcal{F}^{s}$ is a measurable partition of $\Sigma$ and $\Sigma$ is a complete separable metric space. So, Rokhlin's Disintegration Theorem describes a disintegration of $\gamma$ relative to $\mathcal{F}^{s}$ by  a family $\{ \gamma_{\zeta} : \zeta \in \mathcal{F}^{s} \}$, such that, for $E \subset \Sigma$ measurable set,
			\begin{enumerate}
				\item $\gamma_{\zeta}(\zeta)=1$, for $\hat{\gamma}$-almost every $\zeta \in \mathcal{F}^{s}$;
				\item $\zeta \mapsto \gamma_{\zeta}(E)$ is measurable;
				\item $\gamma(E)=\int \gamma_{\zeta}(E) ~d\hat{\gamma}(\zeta)$.
			\end{enumerate}
			Rewrite $\mathcal{F}^{s}$ by $\{ \zeta_{x} \}_{x \in M_{1}}$, where $\zeta_{x}=\{ x \} \times M_{2}$ for each $x \in M_{1}$, and consider $\gamma_{x}'={{\text{proj}}_{M_1}}_{*}\gamma$. For each $x \in M_{1}$, let $\gamma_{\zeta_{x}}$ be the restriction of $\gamma_{\zeta}$ to $\zeta = \{ x \} \times M_{2}$. Note that $\varrho^{-1}(\zeta)= \{ x \} \times M_{2} = \zeta_{x} = {{\text{proj}}_{M_1}}^{-1}(x)$, so
			\begin{align*}
				\int \gamma_{\zeta}(E) d\hat{\gamma}(\zeta) &= \int \gamma_{\zeta}(E) \gamma(\varrho^{-1}(d\zeta)) \\
				&= \int_{M{1}} \gamma_{\zeta_{x}}(E \cap \zeta_{x}) \gamma({{\text{proj}}_{M_1}}^{-1}(dx)) \\
				&= \int_{M{1}} \gamma_{\zeta_{x}}(E \cap \zeta_{x}) d \gamma_{x}'
			\end{align*}
			and then
			\begin{displaymath}
				\gamma(E)=\int_{M{1}} \gamma_{\zeta_{x}}(E \cap \zeta_{x}) d \gamma_{x}'.
			\end{displaymath}
			Moreover, note that $\gamma_{\zeta_{x}}$ is supported in ${{\text{proj}}}_{M_1}^{-1}(x) = \zeta_{x}$. Hence, we have a disintegration $\{\gamma_{\zeta_{x}} : \zeta_{x} \in \mathcal{F}^{s} \}$ along the leaves of $\Sigma$ associated to $\gamma_{x}'$. In addition, it is possible to identify  ${{\text{proj}}}_{M_1}^{-1}(x)$ with $M_{2}$, and write this disintegration by $\{\gamma_{x} : x \in M_1 \} \subset \mathscr{P}(M_{2})$, as desired.	
		\end{example}
	
		\
	
		Such example is one of the possible roles of disintegration of measures. In Dynamics, for instance, this kind of disintegration appears in several contexts. To name a few, in~\cite{BM17} the regularity of this kind of disintegration is investigated considering invariant measures for hyperbolic skew products. Specifically, for this purpose, a function that associates each $x$ in $X$ with a probability measure $\gamma_x$ obtained via a disintegration of $\gamma$ is analysed. In the next section, we will see that this type of application can be thought in a more general framework and it has important properties. We can also cite \cite{Gal17, GL20}, where the disintegration of Example \ref{example1} is used to study the behaviour of the transfer operator in a solenoidal map.
	
		We can actually obtain uniqueness and absolute continuity of the disintegration in the context of Theorem \ref{conditional}. If $\lambda \in \mathcal{M}_{+}(Y)$ is another measure, such that, there exists a Borel map $y \mapsto \eta_{y}$ for which $\mu = \lambda \otimes \eta_{y}$, with $\eta_{y}$ concentrated on $\pi^{-1}(y)$ for $\mu$-almost every $y \in Y$, then $\lambda |_{C} \ll \nu$, where $C=\{ y \in Y : \eta_{y}(X)>0 \}$, and $\eta_{y} \ll \mu_{y}$ for $\nu$-almost every $y \in Y$. See the details in \cite{AP03}. In the following proposition, we focus on the case of product spaces. Taking $\gamma \in \mathcal{M}_{+}(X \times Y)$ with $\mu={\text{proj}_{X}}_{*}\gamma$, we obtain uniqueness of $\gamma_{x}$ and $\mu$ in $\gamma=\mu \otimes \gamma_{x}$.
	
		\
	
		\begin{proposition} \label{uni}
			Let $X \times Y$, $X$ be locally compact and separable metric spaces. Let $\text{proj}_{X}: X \times Y \to X$ be the projection on the first component and consider $\gamma \in \mathcal{M}_{+}(X \times Y)$, $\nu \in \mathcal{M}_{+}(X)$ and $\mu={\text{proj}_{X}}_{*}\gamma$. Let $x \mapsto \eta_x$ be a Borel $\mathcal{M}_{+}$-valued map defined on $X$ such that
			\begin{enumerate}
				\item $\gamma = \nu \otimes \eta_x$;
				\item $\eta_x$  is concentrated on $\text{proj}_{X}^{-1}(x)$ for $\nu$-almost every $x \in X$.
			\end{enumerate}
			Then, the $\eta_{x}$ are uniquely defined $\nu$-almost always by (1) and (2). Moreover, for $C=\{ x \in X: \eta_x(X \times Y)>0 \}$, $\nu|_C$ is absolutely continuous with respect to $\mu$. In particular, $\frac{\nu|_C}{\mu} \eta_x = \gamma_{x}$ for $\mu$-almost every $x \in X$, where $\gamma_{x}$ are the conditional probabilities as in Corollary \ref{desintegra_prod}. 
		\end{proposition}

		\begin{proof}
			Let $\eta_{x}$ and $\eta_{x}'$ be measures satisfying the conditions (1) and (2). Let $(A_{n})_{n}$ be a sequence of open sets such that the finite intersection is also an open set which generates $\mathcal{B}(X \times Y)$, the Borel $\sigma$-algebra of $X \times Y$. Consider $B \in \mathcal{B}(X \times Y)$ and $A=A_n \cap ~{\text{proj}}_{X}^{-1}(B)$, for any $n \in \mathbb{N}$. By the condition (1), we have that
			\begin{equation*}
				\gamma(A)=\int_{X} \eta_{x}(A) d\nu = \int_{B} \eta_{x}(A_n) d \nu
			\end{equation*}
			and
			\begin{equation*}
				\gamma(A)=\int_{X} \eta_{x}' (A) d\nu = \int_{B} \eta_{x}' (A_n) d \nu.
			\end{equation*}
			Therefore,
			\begin{equation*}
				\int_{B} \eta_{x}(A_n) d \nu = \int_{B} \eta_{x}' (A_n) d \nu.
			\end{equation*}
			Given that $B$ is arbitrary, then $\eta_{x}(A_n) = \eta_{x}'(A_n)$ for $\nu$-almost every $x \in X$ and for any $n \in \mathbb{N}$. So, there exists a set $N \subset X$, with $\nu(N)=0$, such that, $\eta_{x}(A_n) = \eta_{x}'(A_n)$  for any $n \in \mathbb{N}$, $x \in (X-N)$. Hence, for $\nu$-almost every $x \in X$, $\eta_{x} = \eta_{x}'$.
	
			Let us denote $C=\{ x \in X : \eta_x(X \times Y)>0 \}$. Let $\mathcal{G} \subset C$ be such that $\mu(\mathcal{G})=0$. So, ${\text{proj}}_{X}^{-1}(\mathcal{G})$ is such that $\gamma ({\text{proj}}_{X}^{-1}(\mathcal{G}))=0$. Therefore, the condition (2) implies
			\begin{equation*}
				0=\int_{X} \eta_{x}({\text{proj}}_{X}^{-1}(\mathcal{G})) d\nu = \int_{\mathcal{G}} \eta_{x} (X \times Y) d\nu.
			\end{equation*}
			Since $\eta_{x} (X \times Y)>0$ in $C \supset \mathcal{G}$, we have that $\nu(\mathcal{G})=0$. That is, $\nu|_{C} \ll \mu$. Moreover, writing $\nu|_{C}=f \mu$ implies that $\gamma = f \mu \otimes \eta_{x}$. But, by the Corollary \ref{desintegra_prod}, $\gamma = \mu \otimes \gamma_{x}$ and then $\frac{\nu|_{C}}{\mu} \eta_{x} = \gamma_{x}$.
		\end{proof}
		
		\
		
	\section{Disintegration maps} \label{sec.disint-maps}

		From the Optimal Transport perspective, Theorem~\ref{conditional}, in fact, deals with the disintegration of transport plans. In this sense, it is possible to define a function from $Y$ to $\mathscr{P}(X)$ with certain properties which actually establishes the link between disintegration of measures and the geometric properties of the measure spaces. We will call such an object the \textbf{disintegration map}. 
	
		\
	
		\begin{definition}[Disintegration map] \label{disint_map}
			Let $X$ and $Y$ be locally compact and separable metric spaces. Consider a measure $\mu \in \mathcal{M}_{+}(X)$, a Borel map $\pi: X \to Y$, and a disintegration of $\mu$ given by Theorem \ref{conditional}, so that $\mu = \nu \otimes \mu_y$. We define the disintegration map:
			\begin{align*}
				f: Y &\to (\mathscr{P}(X), W_{p}) \\
				y & \mapsto \mu_{y},
			\end{align*}
			such that $\mu = \nu \otimes f(y)$, where $W_{p}$ is the p-Wasserstein distance.
		\end{definition}
    
   		\

   		\begin{remark}
   			In order to make clear which measures are associated with the disintegration map, we will say  that ``$f$ is a disintegration map of $\mu$ w.r.t. $\nu$".
   		\end{remark}
   	
   		\

		Although we may define such map to a general $W_{p}$, our main interest will be when either $p=1$ or $p=2$. On the one hand, for $p=1$ there is an explicit formula for the Wasserstein distance which we can apply to study the disintegration of measures in product spaces. On the other hand, when $p=2$, the theory of Optimal Transport allows for a geometric characterization of $\mathscr{P}_{2}(X)$ in terms of the geometric properties of $X$. More precisely, the study of geodesics defined in $\mathscr{P}_{2}(X)$ and the convexity properties of certain functionals along these geodesics play a crucial role in the Metric Theory of Gradient Flows, which allows us to infer geometric properties of $X$ itself. We shall address the case when $p=2$ in the next section. Before that, we use the disintegration maps to further characterize the disintegration of measures in product spaces. In this section, consider the following definition of a disintegration map.
	
		\
	
		\begin{definition}[Disintegration map - product spaces]
			Let $X$ and $Y$ be locally compact and separable metric spaces. Consider $\gamma \in \mathscr{P}(X \times Y)$ and a disintegration of $\gamma$ given by Corollary \ref{desintegra_prod}, so that $\gamma = \mu \otimes \gamma_{x}$. In this case, the disintegration map reads as
			\begin{align*}
				f: X &\to (\mathscr{P}(Y), W_{1}) \\
				x & \mapsto \gamma_{x}
			\end{align*}
			such that $\gamma = \mu \otimes f(x)$.
		\end{definition}
	
		\
	
		In the lines of~\cite{AP03, GM13}, we start by showing the following.
	
		\
	
		\begin{proposition}
			A map $f: X \to (\mathscr{P}(Y), W_{1})$ is a disintegration map if and only if it is Borel.
		\end{proposition}
	
		\begin{proof}		
			Denote by $\text{Lip}_{1}(Y)$ the set of Lipschitz functions whose Lipschitz constants are less or equal $1$. By Ascoli-Arzel\'a Theorem, the space $\text{Lip}_{1}(Y)$ is compact with respect to the uniform convergence \cite[Proposition 3.3.1]{AGS05}. Let $D \subset \text{Lip}_{1}(Y)$ be a countable dense subset and take $\varphi \in \text{Lip}_{1}(Y)$. If the measures $\nu_{1}$, $\nu_{2} \in \mathscr{P}(Y)$ have bounded support, we can use the duality formula to obtain
			\begin{equation*}
				W_{1}(\nu_{1}, \nu_{2}) = \sup\limits_{\varphi \in Lip_{1}(Y)} \Big\{ \int_{Y} \varphi ~d(\nu_{1} -\nu_{2}) \Big\} = \sup\limits_{\varphi \in D} \int_{Y} \varphi ~d(\nu_{1}-\nu_{2});
			\end{equation*}
			see \cite[Section 7.1]{AGS05} for details. For all  $\varphi \in  Lip_{1}(Y)$, 
			\begin{align*}
				\psi_{\varphi}: X &\to \mathbb{R} \\
				x &\mapsto \int_{Y} \varphi ~d(\nu-f(x))
			\end{align*}
			is Borel. For $f$ to be a Borel map it suffices that
			\begin{displaymath}
				f^{-1}(B(\nu, r))= \bigcap\limits_{\varphi \in D} \psi_{\varphi}^{-1}((-r, r)):=A,
			\end{displaymath}
			where $B(\nu, r)$ is an open ball of radius $r$ centred at $\nu$ in $(\mathscr{P}(Y), W_{1})$. In fact, if $x \in A$, then $|\psi_{\varphi}(x)|<r$ for all $\varphi \in D$ and $W_{1}(\nu, f{(x)})<r$ (by the definition of $\psi_{\varphi}$), so that $f{(x)} \in B(\nu, r)$. In the same way, if $f{(x)} \in B(\nu, r)$, then $W_{1}(\nu, f{(x)})<r$ and, by the duality formula,
			\begin{displaymath}
				|\psi_{\varphi}(x)|:=|\int_{Y} \varphi d(\nu-f{(x)})|<r
			\end{displaymath}
			for every $\varphi \in D$, so that $x \in A$. Thus, one way is proven. Conversely, let $A \subset Y$ be an open subset.  Note that 
			\begin{displaymath}
				f(x)(A)=\int_{Y} \mathbbm{1}_A(y) ~df(x).
			\end{displaymath}
			Let $I_{\varphi}$ be a function given by:
			\begin{align}
				I_{\varphi}: (\mathscr{P}(Y), W_{1}) & \to \mathbb{R} \nonumber \\
				\lambda & \mapsto \int_{Y} \varphi(y) ~d\lambda, \label{ifi}
			\end{align} 
			where $\varphi$ is a lower semicontinuous function over $Y$. Since  $W_{1}$ metrises the weak* topology of $\mathscr{P}(Y)$ \cite[Chapter 7]{AGS05}, the function $I_{\varphi}$ is lower semicontinuous. For every $x \in X$, one obtains $\int_{Y} \varphi(y) ~df{(x)}=I_{\varphi}(f(x))$. By assumption, $f$ is Borel and then $f(\cdot)(A):X \to \mathbb{R}$ is a composition of a lower semicontinuous function and a Borel map, so it is a Borel map. Therefore, $f$ is a disintegration map, which concludes the proof.
		\end{proof}
	
		\
	
		The disintegration map can be written in terms of the Monge problem. Given a transport map $T:X \to Y$ and measures $\mu \in \mathscr{P}(X)$, $\nu= T_{*}\mu \in \mathscr{P}(Y)$, the disintegration map is given by $x \mapsto \delta_{T(x)}$, where $\delta_{T(x)}$ is a Dirac measure at $T(x)$. It is possible show that exists a relationship among measures in $(\mathscr{P}(Y), W_{1})$, via push forward of $\mu$ by disintegration maps, and the second marginal of transport plans induced by distinct transport maps.

		\begin{lemma} \label{maps}
			Let $T:X \to Y$ and $S: X \to Y$ be transport maps. Consider a measure $\mu \in \mathscr{P}(X)$ and the applications $f$, and $g$ given by
			\begin{align*}
				f: X &\to (\mathscr{P}(Y), W_{1})\\
				x &\mapsto \delta_{T(x)}
			\end{align*}
			\begin{align*}
				g: X &\to (\mathscr{P}(Y), W_{1}) \\
				x &\mapsto \delta_{S(x)}.
			\end{align*}
			Then, $T_{*}\mu=S_{*}\mu$ if, and only if, $f_{*}\mu=g_{*}\mu$.
		\end{lemma}
	
		\begin{proof}
			Define $\varphi(y)=\psi(\delta_{y})$, where $\psi \in C((\mathscr{P}(Y), W_{1}))$ is arbitrarily chosen. Then,
			\begin{align*}
				\int_{Y} \psi d(f_{*}\mu) &= \int_{X} \psi(f(x)) d\mu \\
				&= \int_{X} \psi (\delta_{T(x)}) d\mu \\
				&= \int_{X} \varphi(T(x)) d\mu \\
				&= \int_{X} \varphi(S(x)) d\mu \\
				&= \int_{Y} \psi d(g_{*}\mu).
			\end{align*}
			Given the arbitrary choice of $\psi$, it follows that $f_{*}\mu=g_{*}\mu$. 
			Conversely, consider $\varphi \in C(Y)$ and the application $I_{\varphi}$ defined by equation \eqref{ifi}. Note that
			\begin{equation*}
				\int_{\mathscr{P}(Y)} I_{\varphi}(\lambda) d(f_{*}\mu)=\int_{\mathscr{P}(Y)} I_{\varphi}(\lambda) d(g_{*}\mu)
			\end{equation*}
			if and only if
			\begin{equation*}
				\int_{X} I_{\varphi}(f{(x)}) d\mu = \int_{X} I_{\varphi}(g{(x)}) d\mu,
			\end{equation*}
			which in turn occurs if and only if
			\begin{equation*}
				\int_{X} \Big( \int_{Y} \varphi (y) df{(x)} \Big) d\mu =  \int_{X} \Big( \int_{Y} \varphi (y) dg{(x)} \Big) d\mu.
			\end{equation*}
			Since
			\begin{displaymath}
				\int_{Y} \varphi (y) ~df{(x)}=\varphi (T(x))
			\end{displaymath}
			and
			\begin{displaymath}
				\int_{Y} \varphi (y) ~dg{(x)}=\varphi (S(x)),
			\end{displaymath}
			the last equation can be written as
			\begin{displaymath}
				\int_{X} \varphi (T(x)) ~d\mu = \int_{X} \varphi (S(x)) ~d\mu.
			\end{displaymath}
			Due to the arbitrary choice of $\varphi$, it follows that $T_{*}\mu=S_{*}\mu$.
		\end{proof}
	
		\
	
		\begin{corollary}
			Let $T:X \to Y$, $S: X \to Y$ be transport maps. Consider a measure $\mu \in \mathscr{P}(X)$, the disintegration maps $f$ and $g$ defined as in the Lemma \ref{maps}, and the measures $\gamma=\mu \otimes f(x)$ and $\eta=\mu \otimes g(x)$. Then, $f_{*}\mu=g_{*}\mu$ if and only if ${\text{proj}_{Y}}_{*}\gamma = {\text{proj}_{Y}}_{*} \eta$. 
		\end{corollary}
	
		\
	
		Given $\gamma, \eta \in \Pi (\mu, \nu)$, as defined in Section~\ref{sec.settings}, with $\gamma= \mu \otimes f{(x)}$ and $\eta=\mu \otimes g{(x)}$, we say that $\gamma$ is \textbf{equivalent by disintegration} to $\eta$ (and we denote $\gamma \approx \eta$) if $f_{*}\mu=g_{*}\mu$. With this equivalence in mind, it is possible to define an equivalent class among the transport plans as follows.
	
		\
	
		\begin{definition}
			Given $\gamma \in \Pi(\mu, \nu)$ with  $\gamma=\mu \otimes f(x)$, the transport class of $\gamma$ is defined as the equivalence class $[\gamma]=\{ \eta=\mu \otimes g(x):g_{*}\mu=f_{*}\mu \}$. 
		\end{definition}
	
		\
	
		Thus, we have that all transport plans induced by transport maps belong to the same transport class. Moreover, in the next proposition, we prove that it is possible to use equivalence by disintegration to assure the existence of a transport map.
	
		\
	
		\begin{proposition}
			Consider a transport map $T:X \to Y$, such that $T_{*}\mu=\nu$ and $\gamma=\mu \otimes \delta_{T(x)}$, for a non-atomic measure $\mu \in \mathscr{P}(X)$. If $\eta \in [\gamma]$, then there exists a transport map $S: X \to Y$, such that, $\eta = \mu \otimes \delta_{S(x)}$.
		\end{proposition} 
	
		\begin{proof}
			By the Approximation Theorem \cite[Theorem 9.3]{Amb00} and by the definition of equivalent by disintegration, there exist a sequence of Borel functions $S_{n}: X \to Y$, such that, $\eta= \lim\limits_{n \to \infty} \mu \otimes \delta_{S_{n}(x)}$ and $(S_{n})_{*}\mu=\nu$ for every $n \in \mathbb{N}$. So $\mu \otimes \delta_{S_{n}(x)} \in [\gamma]$, for every $n \in \mathbb{N}$. Consider $\varphi(y)=|y|^2$ and observe that
			\begin{equation*}
				\int_{X} |S_{n}(x)|^2 ~d\mu = \int_{X} \varphi (S_{n}(x)) ~d\mu = \int_{Y} \varphi(y) ~d\nu = \int_{Y} |y|^2 ~d\nu < \infty.
			\end{equation*}
			
			Let $\psi \in C((\mathscr{P}(Y), W_{1}))$ be a function given by $\psi(\delta_{y})=|y|^2$, and let $\Delta \subset \mathscr{P}(Y)$ be the set of Dirac measures. Observe that $\psi$ is Lipschitz over $\Delta$, with respect to $W_1$. So, take $\psi$ any Lipschitz extension over $\mathscr{P}(Y)$. Since $(\delta_{S_{n}})_{*}\mu=(\delta_{T})_{*}\mu$, we have that
			\begin{equation*}
				\int_{X}|S_{n}(x)|^2 ~d\mu = \int_{X} \psi(\delta_{S_{n}(x)}) ~d\mu = \int_{X} \psi (\delta_{T(x)}) ~d\mu = \int_{X} |T(x)|^{2} ~d\mu
			\end{equation*}
			for every $n \in \mathbb{N}$. Therefore, moving on to a subsequence, we can assume that $S_{n}$ is weakly convergent to $S$. By \cite[Lemma 9.1]{Amb00} $\mu \otimes \delta_{S_{n}(x)}$ converges weakly to $\mu \otimes \delta_{S(x)}$. This concludes the proof.
		\end{proof}
	
		\
	
		In this context, the Monge problem can be interpreted as: minimize $\int_{X \times Y} c(x, y) ~d \gamma$ in a fixed transport class of $\Pi(\mu, \nu)$, that is, obtain
		\begin{displaymath}
			\min \Big\{ \int_{X \times Y} c(x, y) d\gamma : \gamma \in [\mu \otimes \delta_{T}]  \Big\}
		\end{displaymath} 
		for a given transport map $T$. Regarding the Monge-Kantorovich problem, note that the second part of the proof of Lemma~\ref{maps} applies to general transport plans, so that the second marginal can be fixed by the disintegration maps, as follows.
	
		\
	
		\begin{lemma} \label{kantomaps}
			Consider $\mu \in \mathscr{P}(X)$, the disintegration maps $f: X \to \mathscr{P}(Y)$ and $g: X \to \mathscr{P}(Y)$, and the transport plans $\gamma= \mu \otimes f(x)$ and $\eta=\mu \otimes g(x)$. Then, $f_{*}\mu=g_{*}\mu$ implies ${\text{proj}_{Y}}_{*}\gamma = {\text{proj}_{Y}}_{*} \eta$.
		\end{lemma}
	
		\
		
		The reciprocal of Lemma \ref{kantomaps}, however, is not true. Consider, for instance, the following transport problem: Three factories, $x_{1}$, $x_{2}$, and $x_{3}$, with the same production, supply 100\% of their products to two stores, $y_{1}$ and $y_{2}$. Suppose store $y_{2}$ has a demand five times greater than store $y_{1}$. Let $\mu$ be the measure related to the production of the factories and $\nu$ be the measure related to the amount of delivered products to the stores. These measures are given by
		\begin{displaymath}
			\mu=\frac{1}{3} \delta_{x_{1}}+\frac{1}{3} \delta_{x_{2}}+\frac{1}{3} \delta_{x_{3}}
		\end{displaymath}
		\begin{displaymath}
			\nu=\frac{1}{6}\delta_{y_{1}}+\frac{5}{6}\delta_{y_{2}}.
		\end{displaymath}
	
		\
	
		Consider the transport plan illustrated in Figure \ref{transporte_x1}, which divides the products which are produced in the factory $x_{1}$ between the two stores. The corresponding disintegration map $f$ is given by $f(x_1)=\frac{1}{2} \delta_{y_{1}} + \frac{1}{2} \delta_{y_{2}}$, $f(x_2)=\delta_{y_{2}}$ and $f(x_3)=\delta_{y_{2}}$.
	
		\begin{figure}[h!] 
			\centering \includegraphics[scale=0.8]{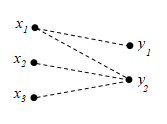}
			\caption{Transport plan 1.} \label{transporte_x1}
		\end{figure}
	
		\
		
		\
		
		\
		
		\

		\noindent Note that, since $\mu$ is of the type $\sum\limits_{i} \alpha_{i} \delta_{x_{i}}$, we have that $f_{*}\mu=  \sum\limits_{i} \alpha_{i} \delta_{f(x_{i})}$. Suppose, however, that there is a change in logistics and in the new transport plan, Figure \ref{transporte_x2}, the factory $x_{1}$ delivers its production only to the store $y_{2}$ and the factory $x_{2}$ divides its production between the two stores. In this new situation, the disintegration map is given by $g(x_1)= \delta_{y_{2}}$, $g(x_2)=\frac{1}{2} \delta_{y_{1}} + \frac{1}{2} \delta_{y_{2}}$ and $g(x_3)=\delta_{y_{2}}$. Nevertheless, on the one hand $f_{*}\mu=g_{*}\mu$, i. e., the transport class does not change. On the other hand, if in this new transport plan the factories  $x_{1}$ and $x_{2}$ deliver to both stores, Figure \ref{transporte_x1x2}, in such a way that $x_{1}$ sends 30 \% of its production to $y_{1}$ and 70\% to $y_{2}$, and $x_{2}$ sends 20 \% of its production to $y_{1}$ and 80\% to $y_{2}$, the disintegration map is given by  $h(x_1) = \frac{3}{10} \delta_{y_{1}} + \frac{7}{10} \delta_{y_{2}}$, $h(x_2) = \frac{2}{10} \delta_{y_{1}} + \frac{8}{10} \delta_{y_{2}}$ and $h(x_3) = \delta_{y_{2}}$. In this case, $f_{*}\mu \neq h_{*}\mu$. So, the transport plans $1$ and $3$ do not belong to the same transport class.
		
		\begin{figure}[h]
			\centering
			\begin{minipage}[c]{6.2cm}	
				\centering
				\includegraphics[scale=0.9]{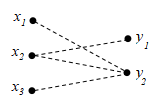}
				\caption{Transport plan 2.} \label{transporte_x2}
			\end{minipage}
			\begin{minipage}[c]{6.2cm}
				\centering \includegraphics[scale=0.8]{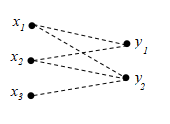}
				\caption{Transport plan 3.}\label{transporte_x1x2}	
			\end{minipage}
		\end{figure}
	
		Furthermore, if there is a small change in transport plan 3, so that $x_{1}$ sends 10\% of its production to $y_{1}$ and 90\% to $y_{2}$, and $x_ {2}$ sends 40\% of its production to $y_{1}$ and 60\% to $y_{2}$, the new disintegration map is given by $k(x_1) = \frac{1}{10} \delta_{y_{1}} + \frac{9}{10} \delta_{y_{2}}$, $k(x_2) = \frac{4}{10} \delta_{y_{1}} + \frac{6}{10} \delta_{y_{2}}$ and $k(x_3) = \delta_{y_{2}}$. Therefore, $h_{*}\mu \neq k_{*}\mu$. Thus, with the previous definition for transport class, the transport class changes when either the number of factories that deliver products to more than one store is changed or even if the fraction of delivered production  is changed.
	
		Therefore, in order to be compatible with the Monge-Kantorovich problem, we need another definition of transport class. Take $\mu \otimes f \in \Pi(\mu, \nu)$ and $\Lambda=f_{*}\mu$, and recall that $(\text{proj}_{Y})_{*}(\mu \otimes f)=\nu$, then for every $\varphi \in C(Y)$,
		\begin{align*}
			\int_{Y}\varphi(y) ~d\nu &=\int_{X} \Big( \int_{Y} \varphi(y) ~df{(x)} \Big) d\mu \\
			&= \int_{X} I_{\varphi}(f{(x)}) ~d\mu \\
			&= \int_{\mathscr{P}(Y)}I_{\varphi}(\lambda) ~d\Lambda(\lambda) \\
			&= \int_{\mathscr{P}(Y)} \Big( \int_{Y} \varphi (y) d\lambda \Big) ~d\Lambda,
		\end{align*}
		where $I_{\varphi}$ is given by equation \eqref{ifi} and $\lambda \in \mathscr{P}(Y)$. Hence, every probability $\Lambda$ in $(\mathscr{P}(Y), W_{1})$ satisfying
		\begin{equation}
			\int_{\mathscr{P}(Y)} \lambda ~d\Lambda=\nu
		\end{equation}
		defines a transport class $[\gamma]=\{\mu \otimes f : f_{*}\mu=\Lambda \}$. As an example, a transport class $[\mu \times \nu]$ corresponds to the measure $\Lambda=\delta_{\nu}$. From this point of view, the Monge-Kantorovich problem in the class $\Lambda$ can be thought as:
		\begin{displaymath}
			MK_{\Lambda}(c, \mu, \nu)=\inf\limits_{\gamma} \Big\{ \int_{X \times Y} c(x, y) ~d\gamma : \gamma = \mu \otimes f, f_{*}\mu=\Lambda \Big\}.
		\end{displaymath}
		This notion of transport class also allows us to see the Monge-Kantorovich problem as an abstract Monge problem between the spaces $X$ and $\mathscr{P}(Y)$, considering the cost
		\begin{displaymath}
			\tilde{c}(x, \lambda)=\int_{Y}c(x,y) ~d\lambda.
		\end{displaymath}
		In fact, for every disintegration map $f: X \to \mathscr{P}(Y)$, such that $f_{*}\mu=\Lambda$,
		\begin{align*}
			\int_{X} \tilde{c}(x, f{(x)}) ~d\mu &= \int_{X} \Big( \int_{Y}c(x,y) df{(x)} \Big) ~d \mu \\
			&= \int_{X \times Y} c(x, y) ~d(\mu \otimes f)
		\end{align*}
		and the problem $MK_{\Lambda}(c, \mu, \nu)$ is equivalent to the Monge problem with parameters $(\tilde{c}, \mu, \Lambda)$. Thus, the disintegration of transport plans introduces another perspective for the optimization problems proposed by Monge and Kantorovich.
	
		\
	
	\section{Absolute continuity from disintegration maps} \label{sec.absolute}

		To connect what we just developed with the geometric properties of the measure spaces and the space of measures, we shall study how paths of measures are given by disintegration maps. Our interest hereafter is to use the $2$-Wasserstein distance.
		
		To construct such paths in $2$-Wasserstein space, it is essential to analyse the weak continuity of the disintegration map.
			In the Wasserstein space, one can study a characterization of convergence. We say that $\{ \mu_k \}$ \textbf{converges weakly} to $\mu$ when
			\begin{displaymath}
				\int \varphi ~\mu_k \longrightarrow \int \varphi ~d\mu
			\end{displaymath}
			for any bounded continuous function $\varphi$. Moreover, Wasserstein distances metrize weak convergence, that is, if $\{ \mu_k \}$ is a sequence in $\mathscr{P}_p (X)$ and $\mu \in \mathscr{P}_p(X)$, then $\mu_k$ converges weakly to $\mu$ is equivalent to $W_p (\mu_k, \mu)\longrightarrow0$. When referring to continuity of the disintegration map, we use \textbf{weak continuity} meaning continuity with respect to weak convergence on $\mathscr{P}_p (X)$.
			
			A map $f$ on an metric space $(Y, \nu)$ is called \textbf{nearly continuous} if, for each $\varepsilon > 0$ there exist $\mathcal{K} \subset Y$ closed with $\nu (Y \backslash \mathcal{K}) < \varepsilon$ such that $f$ restricted to $\mathcal{K}$ is continuous. Let $f$ be a disintegration map of $\mu$ w.r.t. $\nu$. If $\nu$ is absolutely continuous with respect to the volume measure of $Y$, we can apply Lusin's Theorem \ref{thm.Lusin} to show that $f$ is nearly weakly continuous. We also need the following lemma, which is actually a known fact within the Optimal Transport community. For completeness, we add it here in our context.
		
		\
		
		\begin{lemma}\label{lemma.separable}
			Let $(X, d)$ be a separable metric space. The $2$-Wasserstein space $(\mathscr{P}(X), W_{2})$ is a separable metric space.
		\end{lemma}		
		
		\begin{proof}
			Let $\mathcal{D} \subset X$ be a countable and dense subset. Consider the space $\mathcal{M}$ defined by $\mathcal{M} := \{ \nu \in \mathscr{P}(X) : \nu = \sum_j a_j \delta_{x_j}, ~\text{with}~ a_j \in \mathbb{Q} ~\text{and}~ x_j \in \mathcal{D} \}.$ We want to show that $\mathcal{M}$ is dense in $(\mathscr{P}(X), W_2)$. 
			
			Given $\mu \in (\mathscr{P}(X), W_{2})$, then for any $\varepsilon > 0$ and $x_0 \in \mathcal{D}$, there exists a compact set $K \subset X$, such that
			\begin{displaymath}
				\int_{X \backslash K} d(x_0, x)^{2} ~d\mu \leq \varepsilon^{2}.
			\end{displaymath}
			Since $K$ is compact, we may cover it with a family $\{ B(x_k, \frac{\varepsilon}{2}):1 \leq k \leq N,$ $ x_k \in \mathcal{D} \}$	and define 
			\begin{displaymath}
				B_{k}':= B(x_k, \varepsilon) \backslash \bigcup_{j<k} B(x_j, \varepsilon)
			\end{displaymath}
			so that $\{ B_{k}' \}$ are disjoint and still cover $K$. Define $\varphi$ on $X$ by 
			\begin{equation*}
				\begin{cases}
					\varphi(B_{k}' \cap K) = \{ x_k \} \\
					\varphi(X-K)=\{ x_0 \}.
				\end{cases}
			\end{equation*}
			So $d(x, x_k) \leq \varepsilon$ for every $x \in K$. Therefore,
			\begin{displaymath}
				\int_{X} d(x, \varphi(x))^{2} ~d\mu \leq \varepsilon^2 \int_{K} d\mu + \int_{X \backslash K} d(x, x_0)^2 ~d\mu  \leq 2 \varepsilon^2
			\end{displaymath}
			and then $W_2 (\mu, \varphi_{*}\mu) \leq 2 \varepsilon^2$. Note that $\varphi_{*}\mu$ can be written as $\sum_{0 \leq j \leq N} \alpha_j \delta_{x_j}$, that is, $\mu$ might be approximated by a finite combination of Dirac masses. Moreover, the coefficients $\alpha_j$ might be replaced by rational coefficients (up to a small error in Wasserstein distance): since Wasserstein distances are controlled by weighted total variation \cite[Theorem 6.15]{Vil09},
			\begin{displaymath}
				W_{2} \Big( \sum_{0 \leq j \leq N} \alpha_j \delta_{x_j}, \sum_{0 \leq j \leq N} \beta_j \delta_{x_j}  \Big) \leq 2^{\frac{1}{2}}   \Big[ \max_{k, l} d(x_k, x_l) \Big] \sum_{0 \leq j \leq N} |\alpha_j - \beta_j|^{\frac{1}{2}},
			\end{displaymath}
			which can become small, taking suitable coefficients $\beta_j \in \mathbb{Q}$. Thus, it follows that $\mathcal{M}$ is dense in $\mathscr{P}(X)$. Consequently, $(\mathscr{P}(X), W_2)$ is separable. 
		\end{proof}

		\ 
		
		\begin{proposition} \label{weakly_continuous}
			Let $X, ~Y$ be locally compact, complete, and separable metric spaces. Consider $\pi: X \to Y$ a Borel map, $\mu \in \mathcal{M}_{+}(X)$, $\vol_Y$ volume measure on $Y$, and  $\nu := \pi_{*}\mu$. If $\nu \ll \vol_Y$, then the disintegration map of $\mu$ w.r.t. $\nu$ is nearly weakly continuous. 			
		\end{proposition}
		
		\begin{proof}
			Consider the $2$-Wasserstein distance $W_{2}$ on $\mathscr{P}(X)$, with $d$ a complete bounded metric for $X$. $W_{2}$ metrises the weak convergence of $\mathscr{P}(X)$ \cite[Theorem 6.9]{Vil09}. Furthermore, $(\mathscr{P}(X), W_{2})$ is a separable space. Moreover, from Theorem~\ref{conditional}, the map $y  \mapsto \mu_{y} \in \mathscr{P}(X)$ is measurable, and $\nu$ is a Borel measure on $Y$, since $\nu \ll \vol_Y$. Then, by Lusin's Theorem for $Y$ and $\varepsilon > 0$, there exists $\mathcal{K} \subset Y$ with $\nu(Y\backslash \mathcal{K}) < \varepsilon$ such that the disintegration map restricted to  $\mathcal{K}$ is weakly continuous. 
		\end{proof}
		
		\
		
		Although Proposition \ref{weakly_continuous} is relevant by itself, we would like to have conditions for which the disintegration map is weakly continuous at every point, so that, given any two points $y, y'$ in $Y$, we can construct a path of conditional measures connecting $\mu_{y}$ and $\mu_{y'}$. Note that the 2-Wasserstein space is actually connected, and therefore we can always find a path connecting two measures. Nevertheless, we require this path to be specifically given by the disintegration map. This is not trivial indeed, and we can easily construct examples in which this map is not weakly continuous.
		
		\
		
		\begin{example}
			Consider $X = Y = [0,1]$. Let $\mu$ be the Lebesgue measure on $X$ and take the map $\pi: X \to Y$ given by
			\begin{equation*}
				\pi(x) =
				\begin{cases}
					2x, & \text{if~~} x< \frac{1}{2} \\
					1,  & \text{if~~} x \geq \frac{1}{2}.
				\end{cases}
			\end{equation*}
			Note that $\pi$ is Borel measurable, since it is continuous. Define
			\begin{displaymath}
				\nu := \pi_{*} \mu = \frac{1}{2} \lambda + \frac{1}{2} \delta_1
			\end{displaymath}
			where $\lambda$ is the Lebesgue measure on $Y$ and $\delta_1$ is the Dirac measure at $y=1$. A disintegration $\{ \mu_y \}_{y \in Y}$ of $\mu$ with respect to $\pi$ is given by
			\begin{equation*}
				\mu_y =
				\begin{cases}
					\delta_{\pi^{-1}(y)}, & \text{if~~} y<1 \\
					2\lambda|_{\left[\frac{1}{2}, 1\right]}, & \text{if~~} y=1.
				\end{cases}
			\end{equation*}
			In this case, the disintegration map is not weakly continuous at $y=1$. In fact, consider a sequence $(y_n)_n$ in $Y$ such that $y_n \longrightarrow y=1$. Note that
			\begin{align*}
				W_2^2(\mu_{y_n}, \mu_y) =& \inf_{\gamma \in \Pi(\mu_{y_n}, \mu_y)} \int_{X \times X} d(x_1, x_2)^2 ~d \gamma \\
				=& \int_{X \times X} d(x_1, x_2)^2 ~d (\delta_{\pi^{-1}(y_n)} \times \mu_y) \\
				=& \int_{X} d(\pi^{-1}(y_n), x_2)^2 ~d\mu_{y}.
			\end{align*}
			Then, at the limit $y_n \longrightarrow 1$, we have
			\begin{displaymath}
				W_2^2(\mu_{y_n}, \mu_y) \longrightarrow \int_{X} d \Big(\frac{1}{2}, x_2 \Big)^2 ~d\mu_{y} \neq 0.
			\end{displaymath}
			Therefore, $f$ is not weakly continuous at $y=1$. Furthermore, note that $\nu(\{1\}) = \frac{1}{2}$ and $\nu$ is not absolutely continuous with respect to $\lambda$. However, considering $Y=[0, 1)$ and guaranteeing the absolute continuity of $\nu$, we can obtain a good approximation $\mathcal{K}$ in which $f$ is weakly continuous.
		\end{example}
		
		\
		
		We can actually ask for some additional hypotheses so that the disintegration map is weakly continuous. One possibility is to take $\pi$ as a bijective and continuous map.
		
		\
		
		\begin{proposition} \label{bijective}
			Let $X$ and $Y$ be locally compact and separable metric spaces. Consider $\mu \in \mathcal{M}_{+}(X)$, $\pi: X \to Y$ a Borel map, $\nu:= \pi_*\mu$, $\{ \mu_{y} \}_{y \in Y}$ a disintegration of $\mu$ given by Theorem \ref{conditional}, and $f$ the disintegration map of $\mu$ w.r.t. $\nu$. If $\pi$ is bijective and continuous, then $f$ is weakly continuous.  
		\end{proposition}
		
		\begin{proof}
			On the one hand,
			\begin{align*}
				\mu(B) =& \int_{Y} \mu_y (B) ~d\nu \\
				=& \int_{Y} \mu_y (B \cap \pi^{-1}(y)) ~d\nu.
			\end{align*}
			On the other hand,
			\begin{align*}
				\mu (B) =& \int_{X} \mathbbm{1}_{B} ~d\mu \\
				=& \int_{X} \mathbbm{1}_{B} ~d(\nu \circ \pi) \\
				=& \int_{Y} \mathbbm{1}_B (\pi^{-1}(y)) ~d\nu \\
				=& \int_{Y} \mathbbm{1}_B (\pi^{-1}(y)) \delta_{\pi^{-1}(y)} (\pi^{-1}(y)) ~d\nu \\
				=& \int_{Y} \delta_{\pi^{-1}(y)} (B \cap \pi^{-1}(y)) ~d\nu.
			\end{align*}
			Then,
			\begin{equation*}
				\int_{Y} \mu_y (B \cap \pi^{-1}(y)) ~d\nu = \int_{Y} \delta_{\pi^{-1}(y)} (B \cap \pi^{-1}(y)) ~d\nu.
			\end{equation*}
			Moreover, since $\mu_y$ is a probability and $B \cap \pi^{-1}(y)$ is either a singleton or a empty set, we have that 
			\begin{equation*}
				\delta_{\pi^{-1}(y)} (B \cap \pi^{-1}(y)) \geq \mu_y (B \cap \pi^{-1}(y)).
			\end{equation*} 
			Then $\mu_y (B \cap \pi^{-1}(y)) = \delta_{\pi^{-1}(y)} (B \cap \pi^{-1}(y))$, and it follows that $\mu_y = \delta_{\pi^{-1}(y)}$.
			
			\
			
			\noindent Suppose $y_n \longrightarrow y$. Since $\pi$ is continuous, $\pi^{-1}$ is continuous, and then
			\begin{equation*}
				\int g ~d\mu_{y_n} \longrightarrow \int g ~d\mu_y
			\end{equation*}
			for every bounded uniformly continuous function $g$. Therefore, $f$ is weakly continuous.
		\end{proof}
		
		\
		
		Asking for the bijectivity of $\pi$ is quite strong. We want to explore ways to ease this restriction and to obtain continuity at least for almost every point. In some examples which $\pi$ is a quotient map, we have the weak continuity of $f$.
		
		\
		
		\begin{example} \label{ex_prod}
			$\ell_q$-product space from metric measures spaces.
			
			\noindent Let $(Y, d_Y, m_Y)$, $(Z, d_Z, m_Z)$ denote the metric spaces $(Y, d_Y)$ and $(Z, d_Z)$ which are endowed with probability measures $m_Y$ and $m_Z$, respectively. We shall call these metric measure spaces. For $q \in [1, \infty]$, define the $\ell_q$-product space $X := Y \times_{\ell_q} Z$ as the product space $Y \times Z$ equipped with the measure $\mu = m_Y \times m_Z$ and the distance $d_{\ell_q}$ given by
			\begin{equation*}
				d_{\ell_q} ((y, z), (y' z')) =
				\begin{cases}
					[ d_Y (y, y')^q + d_Z (z, z')^q]^{\frac{1}{q}}, & \text{if ~~} 1 \leq q < \infty \\
					\max \{ d_Y(y, y') , d_Z (z, z') \}, & \text{if ~~} q=\infty.
				\end{cases}
			\end{equation*}
			Consider the projection $\pi: X \to Y$ and $\{ \mu_{y} \}_{y \in Y}$ the disintegration of $\mu$ with respect to $\nu:= \pi_* \mu$. Note that
			\begin{displaymath}
				d_{\ell_q} (\pi^{-1}(y), \pi^{-1}(y')) = d_Y (y, y')
			\end{displaymath}
			for every $q \in [1, \infty]$ and, for every conditional measures $\mu_{y}, \mu_{y'}$ 
			\begin{align*}
				W_2^2(\mu_{y}, \mu_{y'}) =& \inf_{\gamma \in \Pi (\mu_{y}, \mu_{y'})} ~ \int d_{\ell_q} ((y_1, z_1), (y_2, z_2))^2 ~ d\gamma \\
				=&  \inf_{\gamma \in \Pi(\delta_{y}, \delta_{y'})} \int d_Y(y_1, y_2)^2 ~d\gamma \\
				=& ~ d_{Y}(y, y')^2.
			\end{align*}  
			Then $W_2 (\mu_{y}, \mu_{y'}) = d_{Y}(y, y')$ and, therefore, the disintegration map is weakly continuous.
		\end{example}
		
		\
		
		In fact, this is one example of a disintegration of measures associated with a foliation, called \textbf{metric measure foliation}, covered in \cite{GKMS18}. For a precise definition, we need to introduce some concepts. Let $(X, d)$ be a metric space. A \textbf{foliation} $\mathcal{F}$ of $X$ is a partition of $X$ into closed subsets. The elements of this partition are called \textbf{leaves}. $\mathcal{F}$ is called a  \textbf{metric foliation} if for every $F, F' \in \mathcal{F}$ and every $x \in F$,
		\begin{displaymath}
			d(F, F') = d(x, F'),
		\end{displaymath}
		where $d(F, F') = \inf \{ d(x, x') : x \in F, x' \in F' \}$ and $d(x, F') = d (\{ x \}, F')$. In case that each leaf is bounded, we say that $\mathcal{F}$ is bounded. Given a metric foliation $\mathcal{F}$ of $X$, define the equivalence relation:
		 \begin{equation} \label{rel_quoti}
		 	 x \sim x' ~\iff~ \exists~ F \in \mathcal{F} ~\text{such that}~ x, x' \in F.
		 \end{equation}
		  Consider the $X^{*} :=X/\sim$ the set of equivalence classes under \eqref{rel_quoti} and the projection  $p: X \to X^{*}$ onto $X^{*}$. We call $X^{*}$ the \textbf{quotient space} and $p$ the \textbf{quotient map}. Define a distance function $d^{*}$ on $X^{*}$ as 
		  \begin{equation}
		  	d^{*}(y, y') := d(p^{-1}(y), p^{-1}(y'))
		  \end{equation}
		  for $y, y' \in X^{*}$. Note that $p$ is a submetry: $p(B(x, r)) = B(p(x), r)$, where $B(x, r)$ is a ball centred at $x$ with radius $r$. In fact, 
		  \begin{align*}
		  	B(p(x), r) =& ~ \{ y \in X^{ *} : d^{*} (y, p(x)) < r\} \\
		  			   =& ~\{ y \in X^{ *} : d (p^{-1} (y), p^{-1}(p(x))) < r\} \\
		  			   =& ~p(B(x, r)).
		  \end{align*}
		Therefore, $p$ is 1-Lipschitz. In this notation, we define:
		
		\
		
		\begin{definition} \label{def_mmf}
			Let $\mathcal{F}$ be a metric foliation of $(X, d, \mu)$. $\mathcal{F}$ is a \textbf{metric measure foliation} if $p_*\mu$ is locally finite Borel measure on $X^*$, and there exists a Borel subset $\Omega \subset X^*$ with $p_*\mu (X^* \backslash \Omega) =0$ such that
			\begin{equation}\label{wass_dis_equal}
				W_2(\mu_y, \mu_{y'}) = d^* (y, y')
			\end{equation}
			for any $y, y' \in \Omega$, where $\{ \mu_y \}_{y \in Y}$ is a disintegration of $\mu$ with respect to $p_{*}\mu$.
		\end{definition}
		
		\
		
		Note that, in this case, the disintegration map is an isometry. This is the most important characteristic of metric measure foliation for us. A very important example of metric measure foliation is related to the action of isometry group.
		
		\
		
		\begin{example} \label{ex_group}
			Let $(X, d, \mu)$ be a metric measure space and $G$ a compact topological group. Let
			\begin{displaymath}
				G \times X \ni (g, x) \mapsto gx \in X
			\end{displaymath}
			be an isometric action of $G$ on $X$. Suppose this action is metric measure isomorphic, that is, for every $g \in G$ the map
			$X \ni x \mapsto gx \in X$ is an isometry preserving the measure $\mu$. Consider $[x]$ the $G$-orbit of a point $x \in X$ and the quotient space $X/G$ endowed with the distance
			\begin{displaymath}
				d_{X/G}([x], [x']) = \inf_{g, g' \in G} d(gx, g'x') .
			\end{displaymath}
			Consider $p: X \to X/G$ the projection map, that is, $p$ is given by $x \mapsto [x]$. The family $\mathcal{F} := \{ p^{-1}(y) : y \in X/G \}$	is a metric measure foliation on $X$. 
		\end{example}
		
		\
		
		Other interesting examples arise from Riemannian submersions of weighted Riemannian manifolds \cite{GKMS18}.
		
		From the Definition \ref{def_mmf} we have a direct association of the $2$-Wasserstein distance between conditional measures (given by Theorem \ref{conditional}) and the distance between points to which these measures were indexed. It is clear that in this case we have the weak continuity of the disintegration map, since it is a isometry. Note that we can relate the $2$-Wasserstein distance between conditional measures to the distance between two points of the quotient space looking at a direction perpendicular to the leaves, through the leaves. Roughly speaking, we can think of a kind of space fibration, where one of the directions has been ``collapsed", and it becomes a parameter for the disintegration family, so that each conditional measure is supported on the underlying fibre (see \figref{mmf_fibre}). That is, the measures $\mu_y$ are of the type $\delta_{y} \times \lambda$, where $\lambda$ is a measure on the fibre, similar to Example \ref{example1}. 
		
		\begin{figure}[h] 
			\includegraphics[scale=0.5]{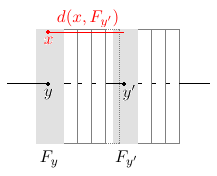} 
			\caption{Idea of a metric measure foliation as a space fibration.}
			\label{mmf_fibre}
		\end{figure}
	
		\
	
		With what we have seen so far we have been able to classify some situation in which the disintegration map is weakly continuous:
		
		\
		\begin{proposition} \label{mmf}
			Let $X$ be locally compact and separable metric space. Consider $\mu \in \mathcal{M}_{+}(X)$, a metric foliation $\mathcal{F}$ of $X$, the quotient space $X^*$, and the quotient map $p:X \to X^*$. If there exists a metric measure foliation of $X$, with $\Omega = X^*$, then the disintegration map of $\mu$ w.r.t. $\nu=p_{*}\mu$ is weakly continuous.
		\end{proposition}
		
		\begin{proof}
			Given a metric measure foliation of $X$, we have for every $y,~ y' \in X^*$, \\
			$W_2(\mu_y, \mu_{y'}) = d^* (y, y')$, where $\mu_{y} = f(y)$ and $\mu_{y'} = f(y')$, and the disintegration map is denoted by $f$. Consider a sequence $(y_n)_n$ in $X^*$ such that $y_n \longrightarrow y$. Note that $W_2(\mu_{y_n}, \mu_{y}) \longrightarrow 0$, since $d^*(y_n, y) \longrightarrow 0$. Therefore we have the weak continuity of $f$.
			
		\end{proof}
		
		\
		
		\begin{remark}
			If in Proposition \ref{mmf} we do not ask for $\Omega = X^*$, the weak continuity of the disintegration map is given for $p_*\mu$-almost every $y \in X^*$ due the definition of metric measure foliation. Although it seems to be a strong condition, in general cases of interest we have it satisfied, as in Examples \ref{ex_prod} and \ref{ex_group}, for instance.
		\end{remark}
	
		\
				
			\begin{remark}
				One can associate hypotheses about the map $\pi$ and the type of disintegration obtained. In light of what we have seen, we may obtain:
				
				\
				
				\begin{enumerate}					
					\item Under the hypotheses of Proposition \ref{bijective}, that is, when the map $\pi$ is bijective and continuous, the conditional measures given by Theorem \ref{conditional} are Dirac deltas.
					
					\item Under the hypotheses of Proposition \ref{mmf}, that is, in the metric measure foliation case, the conditional measures given by Theorem \ref{conditional} are of the type $\delta_{y} \times \lambda$, where $\lambda$ is a measure on the fibre.
				\end{enumerate}
			\end{remark}
		
		\
		
		The entire study carried out on the disintegration map makes clear a fundamental condition for it to be weakly continuous: the supports of the conditional measures $\{ \mu_y \}$ must be disjoint. However, if we want some kind of absolute continuity of $\{ \mu_y \}$ with respect to a reference measure, the supports must have $\mu$-positive measure. In the cases of propositions \ref{bijective} and \ref{mmf}, we do not obtain supports with a $\mu$-positive measure. Another way is to think about the absolute continuity of measures with respect to a reference measure on the fibres. This discussion is summarised in Theorem \ref{thm.absolute}. In the statement, we call $f$ is \textbf{minimizing invariant}, if it maps a minimizing curves on $Y$ to a minimizing curves on $(\mathscr{P}(X), W_2)$. Such a condition is fulfilled when $\pi$ is a Riemannian submersion, for example.

		\
		
		\begin{maintheorem} \label{thm.absolute}
			Let $X$ and $Y$ be locally compact, complete, separable metric spaces. Consider $\pi: X \to Y$ a Borel map, $\mu$ in $\mathcal{M}_{+}(X)$ and  $\nu := \pi_{*}\mu$. If the disintegration map of $\mu$ w.r.t. $\nu$ is weakly continuous and $Y$ is path connected, then given two points $y, y' \in Y$:
			
			\
			
			\noindent (i) there exists a path on $(\mathscr{P}(X), W_{2})$, given by the disintegration map, connecting $\mu_{y}$ and $\mu_{y'}$, the respective conditional measures given by Theorem~\ref{conditional};
			
			\
			
			\noindent (ii) if $X$ is a smooth compact Riemannian manifold equipped with a volume measure $\vol$, $\mu \ll \vol$, $\pi$ is such that $\pi^{-1}(y)$ has $\mu$-positive measure for $\nu$-almost every $y$, the disintegration map is minimizing invariant, and either $\mu_{y}$ or $\mu_{y'}$ is absolutely continuous w.r.t. $\vol$, then all the measures $\mu_{y_{t}}$ on the path given by item (i) are absolutely continuous w.r.t. $\vol$;
			
			\
			
			\noindent (iii) if $\pi$ is such that $\{ \pi^{-1}(y) \}_{y \in Y}$ is a metric measure foliation of $X$, $X$ a smooth compact Riemannian manifold, there exists a path given by the disintegration map connecting $\mu_{y}$ and $\mu_{y'}$, and if either $\mu_{y}$ or $\mu_{y'}$ is absolutely continuous with respect to the volume measure on the respective support fibre, then all the measures $\mu_{y_{t}}$ on this path are absolutely continuous with respect to the volume measure on the fibre.
		\end{maintheorem}
		
		\
		
		\begin{proof}
			
			\
			
			\
			
			\noindent \textit{(i)} This is a direct consequence of the weak continuity of the disintegration map, which will be denoted by $f$ throughout the demonstration. Consider $y, y' \in Y$. Let $\psi$ be a continuous curve in $Y$ connecting $y$ and $y'$, that is, $\psi = \{ y_{t} : t \in [0, 1], ~\psi(0)=y, ~\psi(1)=y' \}$. Taking $y_t \longrightarrow \bar{y} \in Y$, we have $f(y_t) \stackrel{w}{\longrightarrow} f(\bar{y})$, that is, $W_2(\mu_{y_{t}}, \mu_{\bar{y}}) \longrightarrow 0$. Then, $\zeta = \{ \mu_{y_{t}} : t \in [0, 1] \}$, where $\mu_{y_{t}}$ is the conditional measure associated with $y_t$ via $f$, for every $t \in [0, 1]$, is a weakly continuous curve in $(\mathscr{P}(X), W_2)$ connecting $\mu_{y}$ and $\mu_{y'}$. This proves the first part of our theorem.
		
			\
			
			\
			
			\noindent The proof of \textit{(ii)} will be done in a few steps. The idea is to use a sort of ``time-dependent" version of optimal transport. See \cite[Chapter 7]{Vil09}, for example. In short, we will consider the curve $\zeta$ given by the disintegration map as an interpolation between probability measures, called \textbf{displacement interpolation}. To this end, we consider a transport problem, and we associate $\zeta$ with a random curve $\xi$ in $X$.

			\
			
			\noindent \textbf{Step 1:} \textit{A minimising curve in the space of measures.}
			
			\
			
			\noindent Consider the path $\zeta = \{ \mu_{y_{t}} : t \in [0, 1] \}$ as constructed in item \textit{(i)}. Taking $\psi$ as a minimising curve in $Y$, $\zeta$ is a minimising curve in $\mathscr{P}(X)$, since we since $f$ is minimising invariant. By abuse of notation we use the weak continuous path $\zeta: [0,1] \to (\mathscr{P}(X), W_{2})$ while referring to this path in $(\mathscr{P}(X), W_{2})$ given by the disintegration map $f$ evaluated at the minimal curve $\psi$ joining $[y, y']$ in $Y$. More precisely, we consider the composition $f \circ \psi$ to describe $\zeta$.

			\
			
			\noindent \textbf{Step 2:} \textit{A random curve in $X$.}
			
			\
			
			\noindent We want to associate $\zeta$ with a random curve $\xi: [0,1] \to X$. In order to set some notation, let $e_t$ be the \textbf{evaluation map}, given by $e_{t}(\xi) = \xi(t) := \xi_{t}$, meaning the evaluation of $\xi$ at $t$. We will also use some usual concepts related to geodesics in Riemannian manifolds throughout the demonstration. We suggest \cite{Jost11} for a comprehensive reading.

			\
			
			\noindent Consider the curve $\zeta: [0,1] \to \mathscr{P}_2(X)$ joining $\mu_{0}$ and $\mu_{1}$ (from step 1), denoting $\mu_0 = \mu_y$ and $\mu_1 = \mu_{y'}$. We also denote $\mu_{t} = \mu_{y_t}$. Suppose that there is a transport problem associated with this curve, whose respective spatial distributions are modeled by these probability measures. Assume that the cost function for the transport between the initial point $x_0 \in X$ (at time $0$) and the final point $x_1 \in X$ (at time $1$), denoted by $c^{0, 1} (x_0, x_1)$, is associated with a family of functionals parameterized by the initial and the final times. Denote by $\mathcal{A}^{0, 1}$ the functional on the set of curves $[0, 1] \to X$, such that,
			\begin{displaymath}
				c^{0, 1} (x_0, x_1)= \inf \{ \mathcal{A}^{0, 1} (\xi) : \xi_{0} = x_0,~\xi_{1} = x_1, ~\xi \in \mathcal{C}([0, 1]; X) \}.
			\end{displaymath} 
			In other words, $c^{0, 1}(x_0, x_1)$ is the minimal cost needed to go from point $x_0$ at initial time $0$, to point $x_1$ at final time $1$. Moreover, let $C^{0,1}(\mu_{0}, \mu_{1})$ be the optimal transport cost between $\mu_{0}$ and $\mu_{1}$ for the cost $c^{0, 1} (x_0, x_1)$. For $t_1, t_2 \in [0, 1]$, define
			\begin{displaymath}
				\mathcal{A}^{t_1, t_2}(\xi) = \frac{L(\xi)^{2}}{t_2 - t_1}
			\end{displaymath}
			where $L(\xi)$ is the length of $\xi$, so that
			\begin{displaymath}
				c^{t_1, t_2} (x_0, x_1)=\frac{d(x_0, x_1)^2}{t_2 - t_1}
			\end{displaymath}
			and
			\begin{equation}\label{eq.sumcosts}
				C^{t_1, t_2} (\mu_{0}, \mu_{1})  = \frac{W_2 (\mu_{0}, \mu_{1})^2}{t_2 - t_1}.
			\end{equation}
			We want to show that there exists a random minimizer $\xi: [0,1] \to X$, such that, law$(\xi_t) = \mu_t$, for every $t \in [0,1]$. In other words, we want to show that $\zeta$ is a curve in the space of measures which interpolates all possible measures along the minimizing path joining $\mu_{0}$ and $\mu_{1}$. Such a curve is called displacement interpolation.
			
			\
			
			\noindent \textbf{Step 3:} \textit{$\zeta$ is a displacement interpolation.}
			
			\
			
			\noindent In this step, $\xi$ will be constructed by dyadic approximation, according to couplings of measures in $\zeta$ associated with times $t = \frac{1}{2^k}$. In order to achieve this, we will use the iterative construction in lines with \cite[Theorem 7.21]{Vil09}.
			
			\
			
			\noindent Let $\Gamma$ be the set of minimizing curves in $X$. It will be necessary throughout the text, to consider subsets of $\Gamma$ in which the geodesics are defined for certain time intervals and endpoints (or endpoints regions). So, for $s,t \in [0,1]$, $x_s, x_t \in X$, let $\Gamma_{x_s \to x_t}^{s, t}$ be	the set of minimizing curves in $X$ starting at $x_s$ at time $s$, and ending at $x_t$ at time $t$. Similarly, for any two compact sets $K_s, K_t \subset X$, let $\Gamma_{K_s \to K_t}^{s, t}$ be the set of minimizing curves starting in $K_s$ at time $s$, and ending in $K_t$ at time $t$.
			
			\
			
			\noindent Considering the measures along $\zeta$, for $t_1, t_2, t_3 \in [0, 1]$, let $\gamma_{t_1 \to t_2}$ be an optimal transference plan between $\mu_{t_1}$ and $\mu_{t_2}$ for $c^{t_1, t_2}(x_{t_1}, x_{t_2})$, and let $\gamma_{t_2 \to t_3}$ be an optimal transference plan between $\mu_{t_2}$ and $\mu_{t_3}$ for $c^{t_2, t_3}(x_{t_2}, x_{t_3})$. By Lemma \ref{gluing}, it is possible take random variables $(\xi_{t_1}, \xi_{t_2}, \xi_{t_3})$, such that, law$(\xi_{t_1}, \xi_{t_2}) = \gamma_{t_1 \to t_2}$, law$(\xi_{t_2}, \xi_{t_3}) = \gamma_{t_2 \to t_3}$ and law$(\xi_{t_i})=\mu_{t_i}$ for $i=1, 2 ,3$. Since $\zeta$ is minimizing in $\mathscr{P}_2(X)$ (see step 1) and $\mathscr{P}_2(X)$ is a geodesic space, it follows from equation (\ref{eq.sumcosts}), that
			\begin{displaymath}
				C^{t_1, t_2}(\mu_{t_1}, \mu_{t_2}) + C^{t_2, t_3}(\mu_{t_2}, \mu_{t_3}) = C^{t_1, t_3}(\mu_{t_1}, \mu_{t_3}).
			\end{displaymath} 
			This, in particular, implies
			
			\
			
			\noindent (a) $(\xi_{t_1}, \xi_{t_3})$ is an optimal coupling of $(\mu_{t_1}, \mu_{t_3})$ for $c^{t_1, t_3}(\xi_{t_1}, \xi_{t_3})$;
			
			\
			
			\noindent (b) $c^{t_1, t_3} (\xi_{t_1}, \xi_{t_3}) =c^{t_1, t_2} (\xi_{t_1}, \xi_{t_2}) + c^{t_2, t_3} (\xi_{t_2}, \xi_{t_3})$ almost surely.
			
			\
			
			\noindent Let $(\xi_{0}, \xi_{1})$ be an optimal coupling of $(\mu_{0}, \mu_{1})$. Consider optimal transference plans $\gamma_{0 \to \frac{1}{2}}$, $\gamma_{\frac{1}{2} \to 1}$, as above, and construct random variables $\left(\xi_{0}^{(1)}, \xi_{\frac{1}{2}}^{(1)}, \xi_{1}^{(1)}\right)$, such that $\left(\xi_{0}^{(1)}, \xi_{\frac{1}{2}}^{(1)}\right)$ is an optimal coupling of $\left(\mu_{0}, \mu_{\frac{1}{2}}\right)$ for $c^{0, \frac{1}{2}}\left(\xi_{0}^{(1)}, \xi_{\frac{1}{2}}^{(1)}\right)$, $\left(\xi_{\frac{1}{2}}^{(1)}, \xi_{1}^{(1)}\right)$ is an optimal coupling of $\left(\mu_{\frac{1}{2}}, \mu_{1}\right)$ for $c^{\frac{1}{2}, 1}\left(\xi_{\frac{1}{2}}^{(1)}, \xi_{1}^{(1)}\right)$, and law$(\xi_{i}^{(1)}) = \mu_{i}$ for $i=0, \frac{1}{2} ,1$. Moreover, item (a) implies that $(\xi_{0}^{(1)}, \xi_{1}^{(1)})$ is an optimal coupling of $(\mu_{0}, \mu_{1})$, and item (b) implies
			\begin{displaymath}
				c^{0, 1}(\xi_{0}^{(1)}, \xi_{1}^{(1)}) = c^{0, \frac{1}{2}} \left(\xi_{0}^{(1)}, \xi_{\frac{1}{2}}^{(1)} \right) + c^{\frac{1}{2}, 1} \left(\xi_{\frac{1}{2}}^{(1)}, \xi_{1}^{(1)}\right)
			\end{displaymath}
			almost surely. Iterating this process, at the step $k$, we have random variables $\left(\xi_{0}^{(k)}, \xi_{\frac{1}{2^{k}}}^{(k)}, \xi_{\frac{2}{2^{k}}}^{(k)}, \dots, \xi_{1}^{(k)}\right)$, so that for any two $i, j \leq 2^{k}$, $\left(\xi_{\frac{i}{2^{k}}}^{(k)}, \xi_{\frac{j}{2^{k}}}^{(k)}\right)$ is an optimal coupling of $\left(\mu_{\frac{i}{2^{k}}}, \mu_{\frac{j}{2^{k}}}\right)$. Furthermore, for $i_1, i_2, i_3 \leq 2^{k}$,
			\begin{displaymath}
				c^{\frac{i_1}{2^k}, \frac{i_3}{2^k}} \left(\xi_{\frac{i_1}{2^k}}^{(k)}, \xi_{\frac{i_3}{2^k}}^{(k)}\right) = c^{\frac{i_1}{2^k}, \frac{i_2}{2^k}}\left(\xi_{\frac{i_1}{2^k}}^{(k)}, \xi_{\frac{i_2}{2^k}}^{(k)}\right) + c^{\frac{i_2}{2^k}, \frac{i_3}{2^k}}\left(\xi_{\frac{i_2}{2^k}}^{(k)}, \xi_{\frac{i_3}{2^k}}^{(k)}\right)
			\end{displaymath}
			almost always.
			
			\
			
			\noindent We want to extend the random variables $\xi^{(k)}$, defined for times $\frac{i}{2^{k}}$, $i \leq 2^k$, to continuous curves $(\xi^{(k)})_{0 \leq t \leq 1}$. For this, note that for all times $s, t \in [0, 1]$, $s < t$, there exists a Borel map $S_{s \to t}: X \times X \to \mathcal{C}([s, t];X)$, such that for all $x, z \in X$, $S(x, z)$ belongs to $\Gamma_{x \to z}^{s, t}$ \cite[Proposition 7.16]{Vil09}. Indeed, let $E_{s,t}$ be the function given by $E_{s,t}(\xi) := (\xi_{s}, \xi_{t})$, for $\xi: [s, t] \to X$ minimizing curve. Since $X$ is a geodesic space, any two points of $X$ can be joined by at least one minimizing curve, so $E_{s, t}$ is onto $X \times X$. Moreover, $E_{s, t}$ is a continuous map between complete separable metric spaces, and $E_{s,t}^{-1}(x, z)$ is compact for every $x, z$. Therefore, $E_{s,t}$ admits a measurable right-inverse $S_{s \to t}$  \cite{Del75}, that is, $E_{s,t} \circ S_{s \to t} = Id$. Thus $S_{s \to t}$ is a measurable recipe to join two points $x,z$ by a minimizing curve, which was to be proved.
			
			\
			
			\noindent For $t \in \left(\frac{i}{2^k}, \frac{i+1}{2^k}\right)$, define $\xi_{t}^{(k)}$ by $e_t \Big(S_{\frac{i}{2^k} \to \frac{i+1}{2^k}} \left( \xi_{\frac{i}{2^{k}}}, \xi_{\frac{i+1}{2^{k}}} \right) \Big)$. Then, the law of $(\xi_{t}^{(k)})_{0 \leq t \leq 1}$ is a probability on $C(X)$. Let us denote it by $\Theta^{(k)}$. Note that $(e_t)_{*}\Theta^{(k)}=\mu_{t}$ for every $t=\frac{i}{2^k}$, $i \leq 2^k$, and  $\Theta^{(k)}$ is concentrated on $\Gamma$.
			
			\
			
			\noindent In short, from $\zeta$ we iteratively constructed probability measures $\Theta^{(k)}$ on $\Gamma$ up to a step $k$. We want to pass to the limit as $k \longrightarrow \infty$. Given $\varepsilon>0$, since $\mu_{0}$ and $\mu_{1}$ are Radon measures, there exist compact sets $K_0$ and $K_1$ such that $\mu_{0}(X \backslash K_0) \leq \varepsilon$, $\mu_{1} (X \backslash K_1) \leq \varepsilon$. Also, the set $\Gamma_{K_0 \to K_1}^{0, 1}$ is compact \cite[Definition 7.13 and Example 7.15]{Vil09} and
			\begin{align*}
				\Theta^{(k)}\left(\Gamma ~\backslash~ \Gamma_{K_0 \to K_1}^{0, 1}\right) &= \mathbb{P}((\xi_{0}, \xi_{1}) \notin K_0 \times K_1) \\
				& \leq \mathbb{P}(\xi_{0} \notin K_0) + \mathbb{P}(\xi_{1} \notin K_1) \\
				& = \mu_{0}(X \backslash K_0) + \mu_{1} (X \backslash K_1) \\
				& \leq 2 \varepsilon.
			\end{align*}
			Then, we can take a subsequence of $(\Theta^{(k)})_{k}$ that converges weakly to $\Theta$. Since $\Gamma$ is closed in the topology of uniform convergence \cite[Theorem 7.16 (v)]{Vil09}, $\Theta$ is supported in $\Gamma$. Moreover, given a constant $a$, for every $t=\frac{1}{2^{a}} \in [0, 1]$, if $k > a$ we have $(e_t)_{*}\Theta^{(k)} = \mu_{t}$ and, passing to the limit $k \longrightarrow \infty$, $(e_t)_{*}\Theta = \mu_{t}$. Finally, since $\mu_{t}$ depends continuously on $t$, in order to show that $(e_t)_{*}\Theta = \mu_{t}$ for every $t \in [0, 1]$, it suffices to show that $(e_t)_{*}\Theta$ is continuous as a function of $t$. In other words, we need to show that, given $u$ bounded continuous function on $X$, $U(t) = \mathbb{E} u (\xi_{t})$ is a continuous function of $t$ if $\xi$ is random geodesic with law $\Theta$. In fact, since  $t \mapsto \xi_{t}$ is continuous and the composition of continuous functions is also continuous,  $t \mapsto u(\xi_{t})$ is continuous. Moreover, let $\{ u_n \}$ be Lebesgue integrable functions, such that, $u_n \longrightarrow u$. Since $u$ is bounded, $|u_n| \leq g$ for some integrable function $g$; by Lebesgue's Dominated Convergence Theorem $\mathbb{E}u_n \longrightarrow \mathbb{E}u$. From these results the continuity of $U(t)$ follows, as we wanted.
			
			\
			
			\noindent Accordingly, we constructed $\xi$, such that, for each $t \in [0, 1]$, $\mu_{t}$ is the law of $\xi_{t}$, where $(\xi_{t})_{0 \leq t \leq 1}$ is a dynamical optimal coupling of $(\mu_{0}, \mu_{1})$. In other terms, we say that $\zeta$ is displacement interpolation.
			
			\
			
			\noindent \textbf{Step 4:} \textit{An important observation about displacement interpolation.}
			
			\
			
			\noindent By \cite[Theorem 8.5]{Vil09}, if $\{ \mu_t \}$ is a displacement interpolation between two compactly supported probability measures on $X$, and $t_0 \in (0, 1)$ is given, then, for every $t \in [0, 1]$, the transport map $T_{t_0 \to t}$ between the points $\xi(t_0)$ and $\xi(t)$ it is well-defined $\mu_{t_0}$-almost everywhere and it is Lipschitz continuous. In other words, $T_{t_0 \to t}$ is a solution of the Monge problem between $\mu_{t_0}$ and $\mu_t$. For the completeness of the text, we will comment briefly on the proof.
			
			\
			
			\noindent Note that $(e_0, e_1, e_0, e_1)_{*}(\Theta \otimes \Theta) = \gamma_{0 \to 1} \otimes \gamma_{0 \to 1}$. So, if one property holds true $\gamma_{0 \to 1} \otimes \gamma_{0 \to 1}$-almost always for quadruples, this property, for the endpoints of pairs of curves, holds true $\Theta \otimes \Theta$-almost always. Since $\gamma_{0 \to 1}$ is optimal, it has a property named c-cyclical monotonicity \cite[Theorem 5.10]{Vil09}, so that $c(x, y) + c(\tilde{x}, \tilde{y}) \leq c(x, \tilde{y}) + c(\tilde{x}, y)$, $\Theta \otimes \Theta(dx, dy, d\tilde{x}, d\tilde{y})$-a. a. Thus, $c(\xi(0), \xi(1)) + c(\tilde{\xi}(0), \tilde{\xi}(1)) \leq c(\xi(0), \tilde{\xi}(1)) + c(\tilde{\xi}(0), \xi(1))$, $\Theta \otimes \Theta(d\xi, d\tilde{\xi})$-a. a.
			
			\noindent Moreover, by Mather's Shortening Lemma \cite[Theorem 8.1 and Corollary 8.2]{Vil09} 
			\begin{equation} \label{8.6}
				\sup_{0 \leq t \leq 1} d(\xi_t, \tilde{\xi_t}) \leq C_{K} d(\xi_{t_0}, \tilde{\xi_{t_0}}).
			\end{equation}
			where $C_K$ is a constant. Suppose that $\Theta$ is supported on a compact set $S$. The equation \eqref{8.6} defines a closed set for all pairs of curves	$\xi, \tilde{\xi} \in S \otimes S$. 
			Let $e_{t_0}(S)$ be the union of all $\xi(t_0)$, when $\xi$ varies over $S$, and $T_{t_0 \to t}$ by $T_{t_0 \to t}(\xi(t_0)) = \xi(t)$. Note that if $\xi$, $\tilde{\xi}$ in $S$ are such that $\xi(t_0) = \tilde{\xi}(t_0)$, then \eqref{8.6} implies $\xi = \tilde{\xi}$. Moreover, $T_{t_0 \to t}$ is Lipschitz-continuous. So, $(\xi(t_0), T_{t_0 \to t}(\xi(t_0)))$ is a Monge coupling of $(\mu_{t_0}, \mu_{t})$.
			
			\
			
			\noindent \textbf{Step 5:} \textit{Absolute continuity of measures $\mu_t$ on $\zeta$ with respect to the volume measure on $X$, when $\mu_0$ and $\mu_1$ are compactly supported.}
			
			\
			
			\noindent Without loss of generality, let us suppose that $\mu_{1}$ is absolutely continuous with respect to the volume measure on $X$, vol. If $\mu_0$ and $\mu_1$ are compactly supported, then $\zeta$ has a compact support. Indeed, let $A_0, A_1 \subset X$ be the compact supports of $\mu_0, \mu_1$ and $\gamma_{0 \to 1}$ be the transference plan with marginals $\mu_0$ and $\mu_1$. Consider the canonical projections $(\text{proj}_1)$, $(\text{proj}_2)$ on the first and second components, respectively. Since $(\text{proj}_1)_{*}\gamma_{0 \to 1} =\mu_0$  and  $(\text{proj}_2)_{*}\gamma_{0 \to 1} =\mu_1$,
			\begin{displaymath}
				(\text{proj}_1)_{*}\gamma_{0 \to 1}(X \times X) = \mu_{0}(X) = \mu_{0}(A_0) = \gamma_{0 \to 1}(A_0 \times X)
			\end{displaymath}
			\begin{displaymath}
				(\text{proj}_2)_{*}\gamma_{0 \to 1}(X \times X) = \mu_{1}(X) = \mu_{1}(A_1) = \gamma_{0 \to 1}(X \times A_1).
			\end{displaymath} 
			Therefore $\gamma_{0 \to 1}$ is concentrated in a compact set $A_0 \times A_1$. Moreover, since $\gamma_{0 \to 1} = (e_0, e_1)_{*}\Theta$ and the evaluation map is continuous, $\Theta$ is concentrated in a compact set. The compactness of the $\zeta$ support follows from $(e_t)_{*} \Theta = \mu_{t}$.
			
			\
			
			\noindent We can use Step 4, and there is a Lipschitz map $T$ solving the Monge problem between $\mu_{t}$ and $\mu_{1}$, $t \in (0, 1)$. Let $N$ be a set such that the volume measure is zero and consider $T(N)$. If $T(N)$ is not Borel measurable, consider a negligible Borel set that contains $T(N)$ (which by abuse of notation, we will continue denoting $T(N)$). Note that $N \subset T^{-1}(T(N))$, so
			\begin{displaymath}
				\mu_{t}(N) \leq \mu_{t}(T^{-1}(T(N))) = (T_{*} \mu_{t}) (T(N)) = \mu_{1} (T(N))
			\end{displaymath}
			and then $\mu_{t} (N)=0$, since $\text{vol}(T(N)) \leq \| T \|_{\text{Lip}} \text{vol}(N) = 0$ (this inequality occurs since $T$ is Lipschitz and $\| T \|_{\text{Lip}}$ stands for the Lipschitz constant) and $\mu_1 \ll \text{vol}$ by hypothesis. So, $\mu_{t}(N)=0$ for every Borel set $N$ such that vol$(N)=0$, that is, $\mu_{t} \ll \text{vol}$.
			
			\
			
			\noindent \textbf{Step 6:} \textit{Absolute continuity of measures $\mu_t$ of $\zeta$ with respect to the volume measure on $X$: the general case.}
			
			\
			
			\noindent Without loss of generality, suppose $\mu_{1} \ll \text{vol}$. Let us assume that there is some case in which neither $\mu_0$ nor $\mu_1$ is compactly supported. We will prove our statement \textit{(ii)} by contradiction. Suppose that $\mu_{\tau}$, for $\tau \in (0,1)$, is not absolutely continuous with respect to the volume measure on $X$. Then, there exists a set $Z_{\tau} \subset X$, such that vol$(Z_{\tau})=0$ and $\mu_{\tau}(Z_{\tau}) > 0$. Consider $\mathcal{Z}:=\{ \xi \in \Gamma : \xi_{\tau} \in Z_{\tau} \}$ so that $\Theta (\mathcal{Z}) = \mathbb{P}(\xi_{s\tau} \in Z_{\tau}) = \mu_{\tau}(Z_{\tau}) > 0$. Since $\Theta$ is a regular measure, there exists $\mathcal{K} \subset \mathcal{Z}$ compact such that $\Theta(\mathcal{K}) > 0$. So, if we set
			\begin{displaymath}
				\Theta' := \frac{\Theta ~\mathbbm{1}_{\mathcal{K}}}{\Theta(\mathcal{K})}
			\end{displaymath}
			and consider $\gamma_{0 \to 1}':= (e_0, e_1)_{*}\Theta'$ and $\mu_{t}'=(e_t)_{*}\Theta'$, we have
			\begin{displaymath}
				\mu_t' \leq  \frac{(e_t)_{*}\Theta}{\Theta (\mathcal{K})} =\frac{\mu_{t}}{\Theta (\mathcal{K})}.
			\end{displaymath}
			In this way, $(\mu_{t}')$ is a displacement interpolation and, considering the previous equation for $t=1$, $\mu_{1}' \ll \mu_1 \ll \text{vol}$. Note that, now $\mu_{\tau}'$ is concentrated on $e_{\tau}(\mathcal{K}) \subset e_{\tau}(\mathcal{Z}) \subset Z_{\tau}$ and then $\mu_{\tau}'$ is singular. Although $\mu_0'$ is supported in $e_0(\mathcal{K})$ and $\mu_{1}'$ is supported in $e_1(\mathcal{K})$, which are compact. This is the case of Step 5. Then, $\mu_{\tau}' \ll \text{vol}$, which is a contradiction.
			
			\
			
			\
			
			\noindent \textit{(iii)} In this item, we denote $Y:=\Omega$, where $\Omega$ is the subset of the quotient space $X^*$ on which the metric measure foliation is defined (see Definition \ref{def_mmf}), and $\pi := p$, where $p$ is the quotient map. The weak continuity of $f$ in this case was proved in Proposition \ref{mmf}. Let $\psi$ be a minimizing curve on $Y$ connecting $y$ and $y'$. Consider the path $\zeta = \{ \mu_{y_{t}} : t \in [0, 1] \}$ as constructed in item \textit{(i)}. Since $\psi$ was taken as a minimizing curve and $f$ is an isometry, $\zeta$ is minimizing.
			
			\
			
			\noindent Observe that every optimal transference plan between $\mu_y$ and $\mu_{y'}$ is supported on $\{ (x, x') \in \pi^{-1}(y) \times \pi^{-1}(y')  ~:~ d^*(y, y') = d(x, x') \}$. In fact, let $\gamma_{y \to y'}$ be an optimal transference plan for $\mu_y$, $\mu_{y'}$. Since supp$(\mu_y) \subset \pi^{-1}(y)$, supp$(\mu_{y'}) \subset \pi^{-1}(y')$, and $\pi$ is 1-Lipschitz, we have $\gamma_{y \to y'}$ supported on $\{ (x, x') \in \pi^{-1}(y) \times \pi^{-1}(y')  ~:~ d^*(y, y') \leq d(x, x') \}$.
			Consider the set $\Upsilon := \{ (x, x') \in \pi^{-1}(y) \times \pi^{-1}(y')  ~:~ d^*(y, y') < d(x, x') \}$. If $\gamma_{y \to y'} (\Upsilon) >0$, then:
			\begin{align*}
				\gamma_{y \to y'} (\Upsilon) ~d^{*} (y, y')^2 \leq & ~\gamma_{y \to y'} (\Upsilon) ~d(x, x')^2 \\
				< & \int_{\Upsilon} d(x, x')^2 ~d\gamma_{y \to y'}.
			\end{align*}
			So,
			\begin{align*}
				d^{*} (y, y')^2 < & \int_{\Upsilon} d(x, x')^2 ~d\gamma_{y \to y'} + \gamma_{y \to y'} (X \times X \backslash \Upsilon)  ~d^{*} (y, y')^2 \\
				= & \int_{\Upsilon} d(x, x')^2 ~d\gamma_{y \to y'} + \int_{X \times X \backslash \Upsilon} d^*(y, y')^2 ~d\gamma_{y \to y'} \\
				\leq & \int_{\Upsilon} d(x, x')^2 ~d\gamma_{y \to y'} + \int_{X \times X \backslash \Upsilon} d(x, x')^2 ~d\gamma_{y \to y'} \\
				=& \int_{X \times X} d(x, x')^2 ~d\gamma_{y \to y'}.
			\end{align*}
			
			\
			
			\noindent Then, $d^{*} (y, y')^2 < W_2 (\mu_{y}, \mu_{y'})^2$, which is a contradiction.
			
			\
			
			\noindent Therefore, we have the transport between $\mu_y$ and $\mu_{y'}$ orthogonal to the leaves. Furthermore, since $\zeta$ is minimizing, the existence of the optimal transport plan is guaranteed. 
			
			\
			
			\noindent Suppose, without loss of generality that $\mu_y$ is absolutely continuous with respect to the volume measure of the leaf $y$. Then, the support of $\mu_y$ contains more than one point. Considering the transport problem described above, for each $x$ in the support of $\mu_y$ such that $(x, x')$ is in the support of $\gamma_{y \to y'}$ for some $x'$, each one of the intermediate leaves $y_t$ must contain a corresponding point $x_{y_t}$. Thus, each leaf will have the distribution $\mu_{y(t)}$ absolutely continuous with respect to the volume measure of the respective leaf.
		\end{proof}
		
		\
		
		\begin{remark}
			Since $\zeta$ is a displacement interpolation, it is a constant speed geodesic, by \cite[Theorem 2.10]{AG13}. That is, the path in $(\mathscr{P}(X), W_{2})$ given by the disintegration map $f$ evaluated at the minimal curve $\psi$ in $Y$ is a constant speed geodesic in $\mathscr{P}_2(X)$.
		\end{remark}
		
		\
		
		\begin{remark}
			One could, for instance, take the disintegration map in the domain of weak continuity given by Proposition \ref{weakly_continuous}. 
		\end{remark}
		
		\
		
		\begin{remark}
			Theorem \ref{thm.absolute} (iii) holds, for instance, in the case of the disintegration of the volume measure in the Solid Torus (Example \ref{example1}), or in the context of Examples \ref{example_pro}, \ref{ex_prod}, and \ref{ex_group}.
		\end{remark}
		
		\

	\section*{Acknowledgements}
		\noindent This work was financed in part by the Coordena\c{c}\~ao de Aperfei\c{c}oamento de Pessoal de N\'ivel Superior - Brasil (CAPES) - Finance Code 001. R.P. has been partially supported by S\~ao Paulo Research Foundation (FAPESP) - grant \#2018/05309-3 and grant \#2019/14724-7. C. S. R. has been partially supported by S\~{a}o Paulo Research Foundation (FAPESP): grant \#2016/00332-1, grant \#2018/13481-0, and grant \#2020/04426-6. The opinions, hypotheses and conclusions or recommendations expressed in this work are the responsibility of the authors and do not necessarily reflect the views of FAPESP.

		\noindent C. S. R. would like to acknowledge support from the Max Planck Society, Germany, through the award of a Max Planck Partner Group for Geometry and Probability in Dynamical Systems. 

		\noindent The authors are thankful to Pedro Catuogno, Rostislav Matveev, Florentin M\"unch, and Ali Tahzibi for helpful discussions, and to the anonymous referee for valuable suggestions to improve the manuscript.
		
		\
	
	\bibliographystyle{amsalpha}

\begin{thebibliography}{amsalpha}
		
		\bibitem[AG13]{AG13} L. Ambrosio and N. Gigli, \textit{A user's guide to Optimal Transport}, Lecture Notes in Mathematics, vol. 2062, Springer-Verlag, Berlin, 2013.
		
		\bibitem[AGS05]{AGS05} L. Ambrosio,  N. Gigli and G. Savar\'e, \textit{Gradient Flows in Metric Spaces and in the Space of Probability Measures}, Birkh\"auser Verlag, Basel, 2005.
		
		\bibitem[Amb00]{Amb00} L. Ambrosio, \textit{Lecture Notes on Optimal Transport Problems}, Mathematical aspects of evolving interfaces (Lectures notes in Mathematics, 1812), Springer, Berlin, 2000.
	
		\bibitem[AP03]{AP03} L. Ambrosio and A. Pratelli, \textit{Existence and stability results in the L1 theory of optimal transportation}, Lecture Notes in Mathematics, vol. 1813, Springer, Berlin, 2003.
	
		\bibitem[BM17]{BM17} O. Butterley and I. Melbourne, \textit{Disintegration of invariant measures for hyperbolic skew products}, Israel Journal of Mathematics 219, 171-188, 2017.
		
		\bibitem[CP97]{CP97} J. T. Chang and D. Pollard, \textit{Conditioning as disintegration}, Statistica Neerlandica, Vol. 51, nr. 3, pp. 287-317, 1997.
		
		\bibitem[Del75]{Del75} C. Dellacherie. \textit{Ensembles analytiques. Th\'eor\`emes de s\'eparation et applications}, S\'eminaire de Probabilit\'es (Lecture Notes in Mathematics, 465). Ed. P. A. Meyer. Springer, Berlin, 1975.
	
		\bibitem[DM78]{DM78} C. Dellacherie and P. A. Meyer, \textit{Probabilities and potential} (North-Holland Mathematics Studies, 29). North-Holland Publishing Company, Amsterdam, 1978.
	
		\bibitem[Fed69]{Fed69} H. Federer, \emph{Geometric measure theory} (Grundlehren der Mathematischen Wissenschaften, 153). Springer, New York, 1969.	
	
		\bibitem[Gal17]{Gal17} S. Galatolo, \textit{Statistical properties of Dynamics: Introduction to the functional analytic aproach}, 2017. arXiv:1510.02615v2.

		\bibitem[GL20]{GL20} S. Galatolo and R. Lucena, \textit{Spectral gap and quantitative statistical stability for systems with contracting fibers and lorenz-like maps}, Discrete and Continuous Dynamical Systems, 40, 1309-1360, 2020.
		
		\bibitem[GKMS18]{GKMS18} F. Galaz-Garc\'ia, M. Kell, A. Mondino and G. Sosa. \textit{On quotients of spaces with Ricci curvature	bounded below}. Journal of Functional Analysis 275, 1368-1446, 2018.
	
		\bibitem[GM13]{GM13} L. Granieri and F. Maddalena, \textit{Transport problems and disintegration maps}, ESAIM: Control, Optimisation and Calculus of Variations, vol. 19, n. 3, 888-905, 2013. 
		
		\bibitem[Jost11]{Jost11} J. Jost. \textit{Riemannian Geometry and Geometric Analysis}. Springer-Verlag, Berlin, 2011.
	
		\bibitem[OV14]{OV14} K. Oliveira and M. Viana, \textit{Fundamentos da Teoria Erg\'odica}, Sociedade Brasileira de Matem\'atica, Rio de Janeiro, 2014.
		
		\bibitem[Par67]{Par67} K.R. Parthasarathy, \textit{Probability Measures on Metric Spaces}, Academic Press, New York, 1967.
		
		\bibitem[Rok52]{Rok52} V. A. Rokhlin, \textit{On the fundamental ideas of measure theory} (American Mathematical Society Translation, 71), American Mathematical Society, Providence, RI, 1952.
			
		\bibitem[Sim12]{Sim12} D. Simmons, \textit{Conditional measures and conditional expectation; Rohlin's disintegration theorem}, Discrete \& Continuous Dynamical Systems, vol. 32, 2012. 
		
		\bibitem[Stu06]{Stu06} K. T. Sturm, \textit{On the geometry of metric measure spaces}, Acta Math., 196, 65–131, 2006.
		
		\bibitem[Tue75]{Tue75} T. Tjur,\textit{ A Constructive Definition of Conditional Distributions} (Institute of Mathematical Statistics, 13). University of Copenhagen, 1975.
		
		\bibitem[Var16]{Var16} R. Var\~ao. \textit{Center foliation: absolute continuity, disintegration and rigidity}, Ergodic Theory and Dynamical Systems, 36, 256-275, 2016.
		
		\bibitem[Vil03]{Vil03} C. Villani, \textit{Topics in Optimal Transportation} (Graduate Studies in Mathematics, 58). American Mathematical Society, Province, RI, 2003.
	
		\bibitem[Vil09]{Vil09} C.  Villani, \textit{Optimal Transport: old and new} (Grundlehren der Mathematischen Wissenschaften, 338). Springer-Verlag, Berlin, 2009.
	
		\bibitem[Von32]{Von32} J. Von Neumann, \textit{Zur Operatorenmethode in der klassischen Mechanik}. Annals of Mathematics, 33, 587-642, 1932.
	
	
		
		
	\end{thebibliography}
	
\end{document}